\documentclass[a4paper,11pt]{article}
\usepackage[utf8]{inputenc}

\usepackage{boxedminipage}
\usepackage{amsfonts}
\usepackage{amsmath} 
\usepackage{amssymb}
\usepackage{graphicx}
\usepackage{amsthm}
\usepackage{t1enc}
\usepackage{subfig}
\usepackage{cancel}
 \usepackage{textcmds}
 \usepackage{mathabx}

\newtheorem{theorem}{Theorem}[section]
\newtheorem{lemma}[theorem]{Lemma}

\newtheorem{corollary}[theorem]{Corollary}
\newtheorem{definition}[theorem]{Definition}

\newtheorem{fact}[theorem]{Fact}

\newtheorem*{theorem*}{Theorem}

\newcommand{\forceP}{\mathbb{P}}
\newcommand{\forceQ}{\mathbb{Q}}
\newcommand{\forceR}{\mathbb{R}}

\newcommand{\ZFC}{\mathsf{ZFC}}

\newcommand{\ZFP}{\mathsf{ZF}^-}

\newcommand{\CH}{\mathsf{CH}}

\newcommand{\PD}{\mathsf{PD}}

\def\undertilde#1{\mathord{\vtop{\ialign{##\crcr
$\hfil\displaystyle{#1}\hfil$\crcr\noalign{\kern1.5pt\nointerlineskip}
$\hfil\tilde{}\hfil$\crcr\noalign{\kern1.5pt}}}}}

\title{Forcing the ${\Pi^1_3}$-Reduction Property and a Failure of $\Pi^1_3$-Uniformization}
\author{ Stefan Hoffelner\footnote{ WWU M\"unster. Research funded by the Deutsche Forschungsgemeinschaft (DFG German Research Foundation) under Germanys Excellence Strategy EXC 2044 390685587, Mathematics M\"unster: Dynamics-Geometry-Structure. }}
\date{22.09.2022}

\begin{document}

\maketitle

\begin{abstract}
We generically construct a model in which the $\Pi^1_3$-reduction property is true and the $\Pi^1_3$-uniformization property is false, thus producing a model which separates these two principles for the first time.
\end{abstract}

\section{Introduction}
The reduction property was introduced by K. Kuratowski in 1936 and is one of the three regularity properties of subsets of the reals which were extensively studied by descriptive set theorists, along with the separation and the uniformization property. 
\begin{definition}
We say that the $\Pi^1_n$-reduction property holds in a universe $V$ if every pair $A_0,A_1$ of $\Pi^1_n$-subsets of the reals in $V$ can be reduced by a pair of $\Pi^1_n$-sets $D_0,D_1$, which means that $D_0 \subset A_0$, $D_1 \subset A_1$, $D_0 \cap D_1= \emptyset$ and $D_0 \cup D_1=A_0 \cup A_1$.
\end{definition}
The reduction property for $\Pi^1_n$ is implied by the stronger uniformization property for $\Pi^1_n$-sets.
Recall that for an $A \subset 2^{\omega} \times 2^{\omega}$, we say that $f$ is a uniformization (or a uniformizing function) of
$A$ if there is a function $f: 2^{\omega} \rightarrow 2^{\omega}$, 
$dom(f)= pr_1(A)$ (where $pr_1(A)$ is $A$'s projection on the first coordinate) and the graph of $f$ is a subset of $A$. In other words, $f$ chooses exactly one point of every non-empty section of $A$.

\begin{definition}
 We say that the $\Pi^1_n$-uniformization property is true, if 
every set $A \subset 2^{\omega} \times 2^{\omega}$, $A \in \Pi^1_n$ has a uniformizing function $f_A$ whose graph is $\Pi^1_n$.
\end{definition}

Classical work of M. Kondo, building on ideas of Novikov, shows that the $\Pi^1_1$-uniformization (and consequently the $\Sigma^1_2$-uniformization) property is true. This is as much as $\ZFC$ can prove about the uniformization and the reduction property.
In G\"odel's constructible universe $L$, the $\Sigma^1_n$-uniformization-property for $n \ge 3$ is true, as $L$ admits a good $\Sigma^1_2$-wellorder of its reals. On the other hand, due to Y. Moschovakis' celebrated result, the axiom of projective determinacy $\PD$ outright implies the $\Pi^1_{2n+1}$-uniformization property for every $n \in \omega$, indeed ${\Delta}^1_{2n}$-determinacy implies the stronger ${\Pi}^1_{2n+1}$-scale property for every $n \in \omega$ (see \cite{Moschovakis}). By the famous result of D. Martin and J. Steel, $n$-many Woodin cardinals and a measurable above imply $\Pi^1_{n+1}$-determinacy (see \cite{MS} ), in particular under the assumption of $\omega$-many Woodin cardinals $\PD$ becomes true which fully settles the behaviour of the uniformization and reduction property for projective pointclasses.

Despite the extensive list of deep results that has been produced in the last 60 years on this topic, there are still some basic and natural questions concerning the reduction or the uniformization property which remained open.
Note e.g. that in the scenarios above the reduction property holds because the uniformization property does. As these are the only known examples in which the reduction property holds, it is possible that reduction and uniformization for projective pointclasses are in fact equivalent principles over $\ZFC$. 
So the very natural question, which surely has been asked already much earlier, arises whether one can produce universes of set theory where the reduction property holds for some pointclass, yet the corresponding uniformization property fails.
The purpose of our article is to show that this can be done.

\begin{theorem*}
There is a generic extension of $L$ in which the $\Pi^1_3$-reduction property is true and the $\Pi^1_3$-uniformization property is false .
\end{theorem*} 
We expect the arguments to be applicable to the canonical inner models with $n$-many Woodin cardinals, denoted by $M_n$, as well, which would yield models in which the $\Pi^1_{n+3}$-reduction property holds, and the $\Pi^1_{n+3}$-uniformization property fails (see \cite{H2} for a paradigmatic example of how to carefully lift the argument designed for $L$ to work for $M_n$ as well).

This article builds on ideas first introduced in \cite{H} and \cite{H2}. The proof of the theorem is, however, far from a mere application of the two mentioned articles.

The main theme which organises the proof is that the problem of forcing the $\Pi^1_3$-reduction property can be rephrased as a fixed point problem for certain sets of $\aleph_1$-sized proper forcings. This fixed point problem can be solved, which unlocks a seemingly self-referential definition of an iteration which will produce a universe of the $\Pi^1_3$-reduction property. A closer inspection shows that in this universe the $\Pi^1_3$-uniformization property fails.

\section{Preliminaries}

The forcings which we will use in the construction are all well-known. We nevertheless briefly introduce them and their main properties. 

\begin{definition}(see \cite{BHK})
 For a stationary $R \subset \omega_1$ the club-shooting forcing for $R$, denoted by $\forceP_R$ consists
 of conditions $p$ which are countable functions from $\alpha+1 <\omega_1$ to $R$ whose image is a closed set. $\forceP_R$ is ordered by end-extension.
 \end{definition}
The club shooting forcing $\forceP_R$ is the paradigmatic example for an $R$-\emph{proper forcing}, where we say that $\forceP$ is $R$-proper if and only if for every condition $p \in \forceP$, every $\theta > 2^{| \forceP|}$ and every countable $M \prec H(\theta)$ such that $M \cap \omega_1 \in R$ and $p, \forceP \in M$, there is a $q<p$ which is $(M, \forceP)$-generic; and a condition $q \in \forceP$ is said to be $(M,\forceP)$-generic if $q \Vdash ``\dot{G} \cap M$ is an $M$-generic filter$"$, where $\dot{G}$ is the canonical name for the generic filter.  See also \cite{Goldstern}. 
\begin{lemma}
 Let $R\subset \omega_1$ be stationary, co-stationary. Then the club-shooting forcing $\forceP_R$ generically adds a club through $R $. Additionally $\forceP_R$ is $R$-proper, $\omega$-distributive and
 hence $\omega_1$-preserving. Moreover $R$ and all its stationary subsets remain stationary in the generic extension. 
\end{lemma}
\begin{proof}
We shall just show that $\forceP_R$ does not add reals. From the argument one can derive $\omega$-distributivity of $\forceP_R$ easily. The rest can be found in \cite{Goldstern}, Fact 3.5, 3.6 and Theorem 3.7. 

Let $p \in \forceP_R$ and $\dot{x}$ be such that $p \Vdash \dot{x} \in 2^{\omega}$. Without loss of generality we assume that $\dot{x}$ is a nice name for a real, i.e. given by an $\omega$-sequence of $\forceP_R$-maximal antichains. We shall find a real $x$ in the ground model and a condition $q < p$ such that $q \Vdash \dot{x}=x$. For this, fix $\theta > 2^{|\forceP_R|}$ and a countable elementary submodel $M \prec H(\theta)$ which contains $\forceP_R$, $\dot{x}$ and $p$ as elements and which additionally satisfies that $M \cap \omega_1  \in R$. Note that we can always assume that such an $M$ exists by the stationarity of $R$. We recursively construct a descending sequence $(p_n)_{n \in \omega} \subset M$ of conditions below $p=p_0$ such that every $p_n$ decides the value of $\dot{x}(n)$ and such that the sequence of $\operatorname{max}_{n \in \omega} \operatorname{ran}(p_n)$  converges to $M \cap \omega_1$.
We let $x(n) \in 2$ be the value of $\dot{x}$ as forced by $p_n$, and let $x= (x(n))_{n \in \omega} \in 2^{\omega}$.

Let $q'=\bigcup_{n \in \omega} p_n \subset (M \cap \omega_1)$. We set $q:= q' \cup \{ (\operatorname{sup} \operatorname{dom} (q'), M \cap \omega_1) \}$, which is function from $\operatorname{dom} (q') +1$ to $R$ with closed image, and hence a condition in $\forceP_R$, stronger than $p$,  which forces that $\dot{x} =x$ as desired.

\end{proof}

We will choose a family of $R_{\beta}$'s so that we can shoot an arbitrary pattern of clubs through its elements such that this pattern can be read off
from the stationarity of the $R_{\beta}$'s in the generic extension.
For that it is crucial to recall that for stationary, co-stationary  $R \subset \omega_1$, $R$-proper posets can be iterated with countable support and always yield an $R$-proper forcing again. This is proved exactly as in the well-known case for plain proper forcings (see \cite{Goldstern}, Theorem 3.9 and the subsequent discussion).
\begin{fact}
Let $R \subset \omega_1$ be stationary, co-stationary. Assume that $(\forceP_{\alpha} \, : \, \alpha< \gamma)$ is a countable support iteration, let $\forceP_{\gamma}$ denote the resulting partial order obtained using countable support limit and assume also that at every stage $\alpha$, $\forceP_{\alpha} \Vdash  \dot{\forceP}({\alpha})$ is $R$-proper. Then $\forceP_{\gamma}$ is $R$-proper.
\end{fact}

Once we decide to shoot a club through a stationary, co-stationary subset of $\omega_1$, this club will belong to all $\omega_1$-preserving outer models. This hands us a robust method of coding arbitrary information into a suitably chosen sequence of sets
which has been used several times already (see e.g. \cite{SyVera}).
\begin{lemma}\label{coding with stationary sets}
 Let $(R_{ \alpha} \, : \, \alpha < \omega_1)$ be a partition of $\omega_1$ into $\aleph_1$-many stationary sets, let  $r \in 2^{\omega_1}$ be arbitrary, and let $\forceP$ be a countable support iteration $(\forceP_{\alpha} \, : \, \alpha < \omega_1)$, inductively defined via \[\forceP(\alpha) := \dot{\forceP}_{\omega_1 \backslash R_{2 \cdot \alpha}} \text{ if } r(\alpha)=1 \] and
 \[\forceP({\alpha}) := \dot{\forceP}_{\omega_1 \backslash R_{(2 \cdot\alpha) +1}} \text{ if } r(\alpha)=0.\]
 Then in the resulting generic extension $V[\forceP]$, we have that $\forall \alpha < \omega_1:$ \[ r(\alpha)=1 \text{ if and only if }
 R_{2 \cdot \alpha} \text{  is nonstationary, }\]  and \[ r_{\alpha}=0 \text{ iff } R_{(2 \cdot \alpha)+1} \text{ is nonstationary.} \]
\end{lemma}

\begin{proof}
Assume first without loss of generality that $r(0)=1$, then the iteration will be $ R_1$-proper, hence $\omega_1$-preserving. Now let $\alpha < \omega_1$ be arbitrary and assume that $r(\alpha)=1$ in $V[{\forceP}]$. Then by definition of the iteration
 we must have shot a club through the complement of $R_{2 \alpha}$, thus it is nonstationary
 in $V[{\forceP}]$. 
 
 On the other hand, if $R_{2 \alpha}$ is nonstationary in $V[{\forceP}]$, then we assume for a contradiction that we did not use $\forceP_{\omega_1 \backslash R_{ 2 \cdot \alpha}}$ in the iteration $\forceP$.
Note that for $\beta \ne 2 \cdot \alpha$, every forcing of the form $\forceP_{\omega_1 \backslash R_{\beta}}$ is $R_{2 \cdot \alpha}$-proper as $\forceP_{\omega_1 \backslash R_{\beta}}$ is $\omega_1 \backslash R_{\beta}$-proper and $R_{2\cdot \alpha} \subset \omega_1 \backslash R_{\beta}$.
 Hence the iteration $\forceP$ will  be $R_{2 \cdot \alpha}$-proper, thus the stationarity of $R_{2 \cdot \alpha}$ is preserved. But this is a contradiction.

 \end{proof}

The second forcing we use is the almost disjoint coding forcing due to R. Jensen and R. Solovay (see \cite{JensenSolovay}). We will identify subsets of $\omega$ with their characteristic function and will use the word reals for both elements of $2^{\omega}$ as well as for subsets of $\omega$ respectively.
Let $D=\{d_{\alpha} \, \: \, \alpha < \aleph_1 \}$ be a family of almost disjoint subsets of $\omega$, i.e. a family such that if $r, s \in D$ then 
$r \cap s$ is finite. Let $X\subset  \kappa$ for $\kappa \le 2^{\aleph_0}$ be a set of ordinals. Then there 
is a ccc forcing, the almost disjoint coding $\mathbb{A}_D(X)$ which adds 
a new real $x$ which codes $X$ relative to the family $D$ in the following way
$$\alpha \in X \text{ if and only if } x \cap d_{\alpha} \text{ is finite.}$$
\begin{definition}
 The almost disjoint coding $\mathbb{A}_D(X)$ relative to an almost disjoint family $D$ consists of
 conditions $(r, R) \in \omega^{<\omega} \times D^{<\omega}$ and
 $(s,S) \le (r,R)$ holds if and only if
 \begin{enumerate}
  \item $r \subseteq s$ and $R \subseteq S$.
  \item If $\alpha \in X$ and $d_{\alpha} \in R$ then $r \cap d_{\alpha} = s \cap d_{\alpha}$.
 \end{enumerate}
\end{definition}
For the rest of this paper we let $D \in L$ be the definable almost disjoint family of reals one obtains when recursively adding the $<_L$-least real to the family which is almost disjoint from all the previously chosen reals. 
Whenever we use almost disjoint coding forcing, we assume that we code relative to this fixed almost disjoint family $D$.

The last two forcings we briefly discuss are Jech's forcing for adding a Suslin tree with countable conditions and, given a Suslin tree $S$, the associated forcing which adds a cofinal branch through $S$. 
Recall that a set theoretic tree $(S, <)$ is a Suslin tree if it is a normal tree of height $\omega_1$
and has no uncountable antichain. Forcing with a Suslin tree $S$, where conditions are just nodes in $S$, and which we always denote with $S$ again, is a ccc forcing of size $\aleph_1$. 
Jech's forcing to generically add a Suslin tree is defined as follows.

\begin{definition}
 Let $\forceP_J$ be the forcing whose conditions are
 countable, normal trees ordered by end-extension, i.e. $T_1 \le T_2$ if and only
 if $\exists \alpha \le \operatorname{height}(T_1) \, T_2= \{ t \upharpoonright \alpha \, : \, t \in T_1 \}$
\end{definition}
It is well-known that $\forceP_J$ is $\sigma$-closed and
adds a Suslin tree, in fact $\forceP_J$ is forcing equivalent to $\mathbb{C} (\omega_1)$, i.e. the forcing for adding a Cohen subset to $\omega_1$ with countable conditions. The generically added tree $T$ has 
the additional property that for any Suslin tree $S$ in the ground model
$S \times T$ will be a Suslin tree in $V[G]$. This can be used to obtain a robust coding method (see also \cite{Ho} for more applications)
\begin{lemma}\label{1 preservation of Suslin trees}
 Let $V$ be a universe and let $S \in V$ be a Suslin tree. If $\forceP_J$ is 
 Jech's forcing for adding a Suslin tree, if $g \subset \forceP_J$ is generic  and if $T=\bigcup g $ is the generic tree, and if we let $T\in V[g]$ be the forcing which adds an $\omega_1$-branch $b$ through $T$,
 then $$V[g][b] \models  S \text{ is Suslin.}$$
\end{lemma}

\begin{proof}
Let $\dot{T}$ be the $\forceP_J$-name for the generic Suslin tree. We claim that $\forceP_J \ast \dot{T}$ has a dense subset which is $\sigma$-closed. As $\sigma$-closed forcings will always preserve ground model Suslin trees, this is sufficient. To see why the claim is true consider the following set:
$$\{ (p, \check{q}) \, : \, p \in \forceP_J \land height(p)= \alpha+1  \land  \check{q} \text{ is a node of $p$ of level } \alpha \}.$$
It is easy to check that this set is dense and $\sigma$-closed in $\forceP_J \ast \dot{T}$.

\end{proof}

A similar observation shows that a we can add an $\omega_1$-sequence of
such Suslin trees with a countably supported iteration.

\begin{lemma}\label{ManySuslinTrees}
 Let $S$ be a Suslin tree in $V$ and let $\forceP$ be a countably supported
 product of length $\omega_1$ of forcings $\forceP_J$ with $G$ its generic filter. Then in
 $V[G]$ there is an $\omega_1$-sequence of Suslin trees $\vec{T}=(T_{\alpha} \, : \, \alpha \in \omega_1)$ such
that for any finite $e \subset \omega$
the tree $S \times \prod_{i \in e} T_i$ will be a Suslin tree in $V[G]$.
\end{lemma}

These sequences of Suslin trees will be used for coding in our proof and deserve a name, consistent with \cite{FS} and \cite{Ho}.
\begin{definition}
 Let $\vec{T} = (T_{\alpha} \, : \, \alpha < \kappa)$ be a sequence of Suslin trees. We say that the sequence is an 
 independent family of Suslin trees if for every finite set of pairwise distinct indices $e= \{e_0, e_1,...,e_n\} \subset \kappa$ the product $T_{e_0} \times T_{e_1} \times \cdot \cdot \cdot \times T_{e_n}$ 
 is a Suslin tree again.
\end{definition}
\subsection{The ground model $W$ of the iteration}
We have to first create a suitable ground model $W$ over which the actual iteration will take place. $W$ will be a generic extension of $L$ which has no new reals. Moreover $W$ has the crucial property that in $W$ there is an $\omega_1$-sequence $\vec{S}$ of $\omega_1$ trees which is $\Sigma_1(\{\omega_1\})$-definable over $H(\omega_2)^W$ (i.e. the definiton is a $\Sigma_1$-formula with $\omega_1$ as its only parameter) and which forms an independent sequence of Suslin trees in an inner model of $W$\footnote{That $\vec{S}$ is an independent sequence of Suslin trees in an inner model of $W$ only has technical advantages. We have been informed by G. Fuchs that an independent sequence of Suslin trees of length $\omega_1$ actually exists in $L$ already (see \cite{FH}). Thus forcing its existence, as is done in this section, is redundant. Note however that our approach has several advantages, e.g. it is possible to force the existence of independent and definable sequences of Suslin trees of length $\omega_2$ which, of course, can not exist in $L$.  }. The sequence $\vec{S}$ will enable a coding method we will use throughout this article all the time.

To form $W$, we start with G\"odels constructible universe $L$ as our 
ground model.
We first fix an appropriate sequence of stationary, co-stationary subsets of $\omega_1$ as follows.
Recall that $\diamondsuit$ holds in $L$, i.e. over $L_{\omega_1}$ there is a 
$\Sigma_1$-definable sequence $(a_{\alpha} \, : \, \alpha < \omega_1)$ of countable subsets of $\omega_1$
such that any set $A \subset \omega_1$ is guessed stationarily often by the $a_{\alpha}$'s, i.e.
$\{ \alpha < \omega_1 \, : \, a_{\alpha}= A \cap \alpha \}$ is a stationary and co-stationary subset of $\omega_1$.
The $\diamondsuit$-sequence can be used to produce an easily definable sequence of stationary, co-stationary subsets: we list the reals in $L$ in an $\omega_1$ sequence $(r_{\alpha} \, : \, \alpha < \omega_1)$, and let $\tilde{r}_{\alpha} \subset \omega_1$ be the unique element of $2^{\omega_1}$ which copies $r_{\alpha}$ on its first $\omega$-entries followed by $\omega_1$-many 0's. Then, identifying $\tilde{r}_{\alpha} \in 2^{\omega_1}$ with the according subset of $\omega_1$, we define for every $\beta < \omega_1$
a stationary, co-stationary set in the following way:
\[R'_{\beta} = \{ \alpha < \omega_1 \, : \, a_{\alpha}= \tilde{r}_{\beta} \cap \alpha \}.\] It is clear that $\forall \alpha \ne \beta (R'_{\alpha} \cap R'_{\beta} \in \hbox{NS}_{\omega_1})$ and we obtain a sequence of pairwise disjoint stationary sets as usual via setting for every $\beta < \omega_1$ \[R_{\beta}= R'_{\beta} \backslash \bigcup_{\alpha < \beta} R'_{\alpha}.\] and let $\vec{R}=(R_{\alpha} \, : \, \alpha < \omega_1)$. Via picking out one element of $\vec{R}$ and re-indexing we assume without loss of generality that there is a stationary, co-stationary $R \subset \omega_1$, which has pairwise empty intersection with every $R_{\beta} \in \vec{R}$. 
Note that for any $\beta < \omega_1$, membership in $R_{\beta}$ is uniformly $\Sigma_1$-definable over the model $L_{\omega_1}$, i.e. there is a $\Sigma_1$-formula $\psi(x,y)$ such that for every $\beta < \omega_1$
$\alpha \in R_{\beta} \Leftrightarrow L_{\omega_1} \models \psi(\alpha, \beta)$.

We proceed with adding $\aleph_1$-many Suslin trees using of Jech's Forcing $ \forceP_J$. We let 
\[\forceQ^0 = \prod_{\beta \in \omega_1} \forceP_J \] using countable support. This is a $\sigma$-closed, hence proper notion of forcing. We denote the generic filter of $\forceQ^0$ with $\vec{S}=(S_{\alpha} \, : \, \alpha < \omega_1)$ and note that by Lemma \ref{ManySuslinTrees} $\vec{S}$ is independent.  We fix a definable bijection between $[\omega_1]^{\omega}$ and $\omega_1$ and identify the trees in $(S_{\alpha }\, : \, \alpha < \omega_1)$ with their images under this bijection, so the trees will always be subsets of $\omega_1$ from now on.

We work in $L[\forceQ^0]$ and will define the second block of forcings as follows: we let 
\[\forceQ^1= \prod_{\beta < \omega_1} S_{\beta} \]
in other words, we add to each generically created tree from $\vec{S}$ an $\omega_1$-branch, via forcing with the tree. Note that by the argument from the proof of lemma \ref{1 preservation of Suslin trees}, this forcing has a dense subset which is $\sigma$-closed. Hence $L[\forceQ^0][\forceQ^1]$ is a proper and $\omega$-distributive generic extension of $L$.

In a third step we code the trees from $\vec{S}$ into the sequence of $L$-stationary subsets $\vec{R}$ we produced earlier, using Lemma \ref{coding with stationary sets}. It is important to note, that the forcing we are about to define does preserve Suslin trees, a fact we will show later.
The forcing used in the third step will be denoted by $\mathbb{Q}^2$ and will itself be a countable support iteration of length $\omega_1 \cdot \omega_1$ whose components are countable support iteration themselves. First, fix a definable bijection $h \in L_{\omega_2}$ between $\omega_1 \times \omega_1$ and $\omega_1$ and write $\vec{R}$ from now on in ordertype $\omega_1 \cdot \omega_1$ making implicit use of $h$, so we assume that $\vec{R}= (R_{\alpha} \, : \, \alpha < \omega_1 \cdot \omega_1)$. We let $\alpha< \omega_1$ and consider the tree $S_{\alpha} \subset \omega_1$. Defining the $\alpha$-th factor of our iteration $\forceQ^2$, we let $\forceQ^2 (\alpha)$ be the countable support iteration which codes the characteristic function of $S_{\alpha}$ into the $\alpha$-th $\omega_1$-block of the $R_{\beta}$'s just as in Lemma \ref{coding with stationary sets}. So $\forceQ^2 (\alpha)$ is a countable support iteration whose factors, denoted by $\forceQ^2(\alpha) (\gamma)$ are defined via
\[ \forall \gamma < \omega_1 \,(\forceQ (\alpha) (\gamma)= \dot{\forceP}_{\omega_1 \backslash R_{\omega_1 \cdot \alpha + 2 \gamma +1}}) \text{ if } S_{\alpha} (\gamma) =0 \]
and
\[ \forall \gamma < \omega_1 \, (\forceQ( \alpha) (\gamma)= \dot{\forceP}_{\omega_1 \backslash R_{\omega_1 \cdot \alpha + 2 \gamma}}) \text{ if } S_{\alpha} (\gamma) =1. \]

Recall that we let $R$ be a stationary, co-stationary subset of $\omega_1$ which is disjoint from all the $R_{\alpha}$'s which are used. It follows from Lemma \ref{coding with stationary sets} that for every $\alpha < \omega_1$, $\forceQ^2 (\alpha)$ is an $R$-proper forcing which additionally is $\omega$-distributive.  Then we let $\mathbb{Q}^2$ be the countably supported iteration, $$\mathbb{Q}^2=\Asterisk_{\alpha< \omega_1} \forceQ^2 (\alpha)$$ which is again $R$-proper (and $\omega$-distributive as we shall see later).
This way we can turn the generically added sequence of trees $\vec{S}$ into a definable sequence of trees.
Indeed, if we work in $L[\vec{S}\ast \vec{b} \ast G]$, where $\vec{S} \ast\vec{b} \ast  G$ is $\forceQ^0 \ast \mathbb{Q}^1 \ast \forceQ^2$-generic over $L$, then, as seen in Lemma \ref{coding with stationary sets} 
\begin{align*}
\forall \alpha, \gamma < \omega_1 (&\gamma \in S_{\alpha} \Leftrightarrow R_{\omega_1 \cdot \alpha + 2 \cdot \gamma} \text{ is not stationary and} \\ &
\gamma \notin S_{\alpha} \Leftrightarrow  R_{\omega_1 \cdot \alpha + 2 \cdot \gamma +1} \text{ is not stationary})
\end{align*}
Note here that the above formula can be written in a $\Sigma_1(\{\omega_1\} )$-way (i.e. it can be written as a $\Sigma_1$ formula with the ordinal $\omega_1$ as the only parameter), as  it reflects down to $\aleph_1$-sized, transitive models of $\ZFP$ which contain a club through exactly one element of every pair $\{(R_{\alpha}, R_{\alpha+1}) \, : \, \alpha < \omega_1\}$.

Our goal is to use $\vec{S}$ for coding. For this it is essential, that the sequence remains independent in the inner  universe $L[\forceQ^0 \ast \forceQ^2]$. Note that this is reasonable as $\forceQ^0 \ast \forceQ^1 \ast \forceQ^2$ can be written as $\forceQ^0 \ast ( \forceQ^1 \times \forceQ^2)$, hence one can form the inner model $L[\forceQ^0 \ast \forceQ^2]$ without problems.

The following  line of reasoning is similar to \cite{Ho}.
Recall that for a forcing $\forceP$ and $M \prec H(\theta)$, a condition $q \in \forceP$ is $(M,\forceP)$-generic iff for every maximal antichain $A \subset \forceP$, $A \in M$, it is true that $ A \cap M$ is predense below $q$.
The key fact is the following (see \cite{Miyamoto2} for the case where $\forceP$ is proper)
\begin{lemma}\label{preservation of Suslin trees}
 Let $T$ be a Suslin tree, $R \subset \omega_1$ stationary and $\forceP$ an $R$-proper
 poset. Let $\theta$ be a sufficiently large cardinal.
 Then the following are equivalent:
 \begin{enumerate}
  \item $\Vdash_{\forceP} T$ is Suslin
 
  \item if $M \prec H_{\theta}$ is countable, $\eta = M \cap \omega_1 \in R$, and $\forceP$ and $T$ are in $M$,
  further if $p \in \forceP \cap M$, then there is a condition $q<p$ such that 
  for every condition $t \in T_{\eta}$, 
  $(q,t)$ is $(M, \forceP \times T)$-generic.
 \end{enumerate}

\end{lemma}

\begin{proof}
For the direction from 1 to 2 note first that $\Vdash_{\forceP} ``T$ is Suslin$"$ implies $\Vdash_{\forceP} ``T$ is ccc$"$, and in particular for any countable elementary submodel $N[\dot{G}_{\forceP}] \prec H(\theta)^{V[\dot{G}_{\forceP}]}$,  $\Vdash_{\forceP} \forall t \in T \, (t$ is $(N[\dot{G}_{\forceP}],T)$-generic). Now if $M \prec H(\theta)$ and $M \cap \omega_1 = \eta \in R$ and $\forceP,T \in M$ and $p \in \forceP \cap M$ then there is a $q<p$ such $q$ is $(M,\forceP)$-generic. So $q \Vdash \forall t \in T \, (t$ is $(M[\dot{G}_{\forceP}], T)$-generic, and this in particular implies that $(q,t)$ is $(M, \forceP \times T)$-generic for all $t \in T_{\eta}$. 

For the direction from 2 to 1 assume that $\Vdash \dot{A} \subset T$ is a maximal antichain. Let $B=\{(x,s) \in \forceP \times T \, : \, x \Vdash_{\forceP} \check{s} \in \dot{A} \}$, then $B$ is a predense subset in $\forceP \times T$. Let $\theta$ be a sufficiently large regular cardinal and let $M \prec H(\theta)$ be countable such that $M \cap \omega_1=\eta \in R$ and $\forceP, B,p,T \in M$. By our assumption there is a $q <_{\forceP} p$  such that $\forall t \in T_{\eta} \, ((q,t)$ is $(M, \forceP \times T)$-generic). So $B \cap M$ is predense below $(q,t)$ for every $t \in T_{\eta}$, which yields that $q \Vdash_{\forceP} \forall t \in T_{\eta}\,  \exists s<_{T} t \, (s \in \dot{A})$ and $s$ can be found in $M$, hence $q \Vdash \dot{A} \subset T \upharpoonright \eta$, so $\Vdash_{\forceP} T$ is Suslin.
\end{proof}
In a similar way, one can show that Theorem 1.3 of \cite{Miyamoto2} holds true if we replace proper by $R$-proper for $R \subset \omega_1$ a stationary subset.
\begin{theorem}
Let $(\forceP_{\alpha})_{\alpha < \eta}$ be a countable support iteration of length $\eta$, let $R \subset \omega_1$ be stationary and suppose that for every $\alpha < \eta$, for the $\alpha$-th factor of the iteration $\dot{\forceP}(\alpha)$ it holds that $\Vdash_{\alpha}  ``\dot{\forceP}(\alpha)$ is $R$-proper and 
preserves every Suslin tree.$"$ Then $\forceP_{\eta}$ is $R$-proper and preserves every Suslin tree.
\end{theorem}
So in order to argue that our forcing $\forceQ^2$ preserves Suslin trees when used over the ground model $W[\forceQ^0]$, it is sufficient to show that every factor preserves Suslin trees.
This is indeed the case.
\begin{lemma}
Let $R \subset \omega_1$ be stationary, co-stationary, then the club shooting forcing $\forceP_R$ preserves Suslin trees.
\end{lemma}

\begin{proof}
Because of Lemma \ref{preservation of Suslin trees}, it is enough to show that for for any Suslin tree $T$, any regular and sufficiently large
 $\theta$, every $M \prec H_{\theta}$ with $M \cap \omega_1 = \eta \in R$, and every
 $p \in \forceP_R \cap M$ there is a $q<p$ such that for every
 $t \in T_{\eta}$, $(q,t)$ is $(M,(\forceP_R \times T))$-generic.
 Note first that as $T$ is Suslin, every node $t \in T_{\eta}$ is an
 $(M,T)$-generic condition. Further, as forcing with a Suslin tree
 is $\omega$-distributive, $M[t]$ has the same $M[t]$-countable sets as $M$.
 By the argument of the proof of Lemma 2.2, if $M\prec H(\theta)$ is such
 that $M \cap \omega_1 \in R$ then an $\omega$-length descending sequence
 of $\forceP_R$-conditions in $M$ whose domains converge to $M \cap \omega_1$
 has a lower bound as $M \cap \omega_1 \in R$.
 
 We construct an $\omega$-sequence of elements of $\forceP_R$ which has a lower bound
 which will be the desired condition. 
 We list the nodes on $T_{\eta}$, $(t_i \, : \, i \in \omega)$ and
 consider the according generic extensions $M[t_i]$.
 In every $M[t_i]$ we list the $\forceP_R$-dense subsets of $M[t_i]$,
 $(D^{t_i}_n \, : \, n \in \omega)$ and write 
 the so listed dense subsets of $M[t_i]$ as an $\omega \times \omega$-matrix and enumerate
 this matrix in an $\omega$-length sequence of dense sets $(D_i \, : \, i \in \omega)$.
 If $p=p_0 \in \forceP_R \cap M$ is arbitrary we can find, using the fact that $\forall i \, (\forceP_R \cap M[t_i] = M \cap \forceP_R$), an $\omega$-length, descending
 sequence of conditions below $p_0$ in $\forceP_R \cap M$, $(p_i \, : \, i \in \omega)$
 such that $p_{i+1} \in M \cap \forceP_R$ is in $D_i$.
 We can also demand that the domain of the conditions $p_i$ converge to $M \cap \omega_1$.
 Then the $(p_i)$'s have a lower bound $p_{\omega} \in \forceP_R$ and $(t, p_{\omega})$ is an
 $(M, T \times \forceP_R)$-generic conditions for every $t \in T_{\eta}$ as any $t \in T_{\eta}$ is $(M,T)$-generic
 and every such $t$ forces that $p_{\omega}$ is $(M[T], \forceP_R)$-generic; moreover $p_{\omega} < p$ as 
 desired.
 \end{proof}
 
Let us set $W:= L[\forceQ^0\ast \forceQ^1  \ast \forceQ^2 ]$ which will serve as our ground model for a second iteration of length $\omega_1$. To summarize the above:

\begin{theorem}
The universe $W= L[\forceQ^0 \ast \forceQ^1 \ast \forceQ^2]$ is an $\omega$-distributive generic extension of $L$, in particular no new reals are added and $\omega_1$ is preserved. In $W$ there is a $\Sigma_1(\{\omega_1\})$-definable, independent sequence of  trees $\vec{S}$ which are Suslin in the inner model $L[\forceQ^0][\forceQ^2]$, yet no tree is Suslin in $W$. 
\end{theorem}
\begin{proof}
The first assertion should be clear from the above discussion. The second assertion holds by the following standard argument. As $\forceQ^0 \ast \forceQ^1$ does not add any reals it is sufficient to show that $\forceQ^2$ is $\omega$-distributive in $L[\forceQ^0][\forceQ^1]$.
Let $p \in \forceQ^2$ be a condition and assume that $p \Vdash ``\dot{r} $ is a countable sequence of ordinals$"$. We shall find a stronger $q < p$ and a set $r$ in the ground model such that $q \Vdash \check{r}=\dot{r}$. Let $M \prec H(\omega_3)$ be a countable elementary submodel which contains $p, \forceQ^2$ and $\dot{r}$ and such that $M \cap \omega_1 \in R$, where $R$ is our fixed stationary set from above. Inside $M$ we recursively construct a decreasing sequence $p_n$ of conditions in $\forceQ^2$, such that for every $n$ in $\omega,$ $p_n \in M$, $p_n$ decides $\dot{r}(n)$ and for every $\alpha$ in the support of $p_n$, the sequence $\operatorname{sup}_{n \in \omega} \operatorname{max}( p_n(\alpha))$ converges towards $M \cap \omega_1$ which is in $R$. Now, $q':= \bigcup_{n \in \omega} p_n$ and for every $\alpha< \omega_1$ such that $q'(\alpha)\ne 1$ (where 1 is the weakest condition of the forcing),  in other words for every $\alpha$ in the support of $q'$ we define $q(\alpha):= q'(\alpha) \cup \{(\omega,sup (M \cap \omega_1))\}$ and $q(\alpha)=1$ otherwise. Then $q=(q(\alpha))_{\alpha < \omega_1}$ is a condition in $\forceQ^2$, as can be readily verified and $q \Vdash \dot{r} = \check{r}$, as desired.
\end{proof}
The independent sequence $\vec{S}$ will be split into two $\Sigma_1(\{\omega_1\})$-definable sequences via letting
\[\vec{S}^1= ( S_{\alpha} \in \vec{S} \, : \, \alpha \text{ is even}) \] and
\[ \vec{S}^2= (  S_{\alpha} \in \vec{S} \, : \, \alpha \text{ is odd}). \]
These two sequences will be used for defining the $\Pi^1_3$-sets witnessing the reduction property, as we will see soon.

We end with a straightforward lemma which is used later in coding arguments.

\begin{lemma}\label{a.d.coding preserves Suslin trees}
 Let $T$ be a Suslin tree and let $\mathbb{A}_D(X)$ be the almost disjoint coding which codes
 a subset $X$ of $\omega_1$ into a real with the help of an almost disjoint family
 of reals $D$ of size $\aleph_1$. Then $$\mathbb{A}_{D}(X) \Vdash_{} T \text{ is Suslin }$$
 holds.
\end{lemma}
\begin{proof}
 This is clear as $\mathbb{A}_{D}(X)$ has the Knaster property, thus the product $\mathbb{A}_{D}(X) \times T$ is ccc and $T$ must be Suslin in $V[{\mathbb{A}_{D}(X)}]$. 
\end{proof}

\section{Main Proof}
\subsection{Informal discussion of the idea}
We proceed with an informal discussion of the main ideas of the proof. We focus on reducing one fixed, arbitrary pair $A_m$ and $A_k$ of $\Pi^1_3$-sets. The arguments will be uniform, so that reducing every pair of $\Pi^1_3$-sets will follow immediately.

The ansatz is to use the two definable sequences of Suslin trees $\vec{S}^1$ and $\vec{S}^2$ for coding and a bookkeeping function $F$ which lists all possible reals in our iteration. We use an iterated forcing construction over $W$ of length $\omega_1$. At stages $\beta$, where $F(\beta)$ is (the name of) a real number $x$, we decide whether to code $x$ into the $\vec{S}^1$-sequence or the $\vec{S^2}$-sequence. Coding here means that we write the characteristic function of $x$ into $\aleph_1$-many $\omega$-blocks of elements of $\vec{S}^i$, $i \in \{1,2\}$ in a way such that the statement \ldq $x$ is coded into $\vec{S}^i$\rdq\, is a $\Sigma^1_3(x,i)$-statement and hence a $\Sigma^1_3(x)$-statement. We will see later that the definition of iterations of our coding forcing will only depend on the set of reals which are coded, and that the set of our coding forcings is closed under products.
Our goal is that eventually, after $\omega_1$ stages of our iteration, the resulting universe satisfies 
\begin{align*}
(1) \qquad \forall x \in A_m \cup A_k,\text{ either $x$ is coded into $\vec{S^1}$ or $x$ is coded into $\vec{S^2}$. } 
\end{align*} 
This dichotomy has to be strengthened to produce the desired reducing sets for $A_m$ and $A_k$ as follows:
We shall aim for a universe in which the set of reals of $A_m \cup A_k$ which are \emph{not} coded  into $\vec{S^1}$,  will be a subset of $A_m$, and the set of reals which are not coded into $\vec{S}^2$ will be a subset of $A_k$. Assuming we can pull this off, we obtain the following equality
\begin{align*}
(2) \qquad D^1_{m,k}:= & \{ x \in A_m  \, : \, x \text{ is  not coded into $\vec{S^1}$} \}  = \\& \{ x \in A_m  \cup A_k \, : \, x \text{ is coded  into $\vec{S^2}$} \}
\end{align*}
Note that the definition of $D^1_{m,k}$ witnesses that $D^1_{m,k}$ is $\Pi^1_3$, as being coded is $\Sigma^1_3$, hence not being coded is $\Pi^1_3$.

On the other hand, our universe should satisfy that reals in $A_k$ which are \emph{not} coded into $\vec{S^2}$ form a set  $D^2_{m,k}$ which eventually should reduce $A_k$:
\begin{align*}
(3) \qquad D^2_{m,k}:= & \{ x \in  A_k \, : \, x \text{ is  not coded  into $\vec{S^2}$} \} \\=& \{ x \in A_m \cup A_k \, : \, x \text{ is coded  into $\vec{S^1}$} \} 
\end{align*}
If we could achieve a universe for which (1), (2) and (3) holds, then $D^1_{m,k} \cup D^2_{m,k} = A_m \cup A_k$, the sets are $\Pi^1_3$, $D^1_{m,k} \subset A_m$ and $D^2_{m,k} \subset A_k$ and $D^1_{m,k} \cap D^2_{m,k} = \emptyset$, i.e. there are reducing sets for $A_m$ and $A_k$.
 
This set-up has the following difficulties one has to overcome: the evaluation of $\Pi^1_3$-sets changes as we use coding forcings, yet deciding to code a real into the, say, $\vec{S}^1$-sequence, once performed, can not be undone in future extensions, by the upwards absoluteness of $\Sigma^1_3$-formulas.

For an illustration of how this can be problematic, assume that $x \in L$ is a real which is in both $A_m$ and $A_k$. Now we need to decide, where to put $x$. As we have not settled on a strategy yet, how to do it, we simply
assume that we code the real $x$ into $\vec{S^2}$, which is equivalent to put $x$ into $D^1_{m,k}$. Now it could happen that this coding forcing, or some later coding forcing we will use, actually puts $x$ out of $A_m$, while $x$ remains in $A_k$. The consequence of this is that $x \in A_m \cup A_k$ witnesses that $D^1_{m,k}$ and $D^2_{m,k}$ do not reduce $A_m$ and $A_k$, as now $x \notin A_m$, hence $x\notin D^1_{m,.k}$ and $x$ is coded into $\vec{S}^2$, hence $x$ is not an element of $D^2_{m,k}$ either.

In the above situation, two possibilities arise, which already hint in a simplified setting how we eventually want to solve the problem of forcing the $\Pi^1_3$-reduction property.
First it could happen that there is an additional iteration of coding forcings which forces $x$ out of $A_k$. This would repair the misery, as after we did this additional forcing, $x$ is not an element of $A_m \cup A_k$ anymore and does not play a role in our task of reducing $A_m$ and $A_k$. Note however, that this additional forcing could possibly create new problems for other reals than $x$; indeed it could happen that we will create a non-empty set of reals $y \ne x$ for which the same pathological situation as for $x$ arises. Thus in order to reduce $A_m$ and $A_k$, we need  to ensure that the forcing we use to repair the situation for $x$ will not create new problems for other reals. This problem, if interpreted from the right angle, is a fixed point problem for sets of coding forcings. We will see later that solving this fixed point problem hands us the right set of coding forcings, which can safely force the real $x$ out of $A_m \cup A_k$ without adding unrepairable further pathological situations for other reals $y \ne x$.

In the second possibility, there is no such iterated coding forcing which forces $x$ out of $A_k$. 
But this also tells us, that instead of coding $x$ into $\vec{S}^2$, we could have coded $x$ into $\vec{S}^1$, that is we could have put $x$ into $D^2_{m,k}$ and $x$ will remain in $A_k$ after our coding, as otherwise we would have an iterated coding forcing which forces $x \notin A_k$, which we assumed to not exist. 
Note that this argument uses the upwards absoluteness of $\Sigma^1_3$-formulas, hence Shoenfield absoluteness, and the earlier mentioned fact that coding forcings are closed under taking products.

Thus, in the second case, the absence of a way to repair the bad situation for our real $x$ in fact yields valuable information. It tells us that we can safely code $x$ into the other of the two definable sequences of $\omega_1$-trees, namely $\vec{S^1}$, and $x$ will remain in $A_k$, as long as we just stick with iterating the coding forcings, which is what we will do anyway. So, assuming the second case applies, at least for the one real $x$ and the two fixed $\Pi^1_3$-sets $A_m$ and $A_k$ we found an unproblematic way of coding $x$.

Applying the above reasoning for every real we encounter in our iteration will lead to a new set of rules how to form an iteration of coding forcings, called 1-allowable, which, when applied, will yield a more reasonable class of coding forcings to produce a model for which (1), (2)  and (3) holds. For the set of 1-allowable forcings, the above mentioned pathological situations can arise again, but as before, we can use these newly arising problems to define a new set of rules, yielding a new set of forcings and so on. This leads to an inductive definition of an operator which acts on sets of iterated coding forcings. The operator, which can be seen as some sort of derivation, will be uniformly definable over $W$.
We will set things up in such a way that the process of applying the operator transfinitely often converges to a unique and non-empty set of forcings, which we will denote the $\infty$-allowable forcings. All of these considerations are preliminary. 

Eventually,  we can use the fact that $\infty$-allowable forcings form a fixed point to define an iteration in a seemingly self-referential way which, while iteratively producing the reducing sets $D ^1_{m,k}$ and $D^2_{m,k}$, will take care of all the pathological situations which arise along its definition.  
To be a bit more precise, if $\mathcal{P}$ denotes the class of $\infty$-allowable forcings, we can define an iteration as follows:  for any real $x$ and any pair of $\Pi^1_3$-sets $A_m$ and $A_k$, either there is a forcing from $\mathcal{P}$ which forces $x \notin A_m \cup A_k$ in which case we opt to use such a forcing from $\mathcal{P}$. Or, for $m$ or $k$, the real $x$ can not be forced out of $A_m$ or $A_k$ with a forcing from $\mathcal{P}$, in which case we code $(x,m,k)$ into the corresponding sequence of $\omega_1$-trees $\vec{S}^1$ or $\vec{S}^2$ to create our reducing sets. Most importantly, the so defined iteration is an element of $\mathcal{P}$ itself, which will ultimately imply that one can force  the $\Pi^1_3$-reduction property. 

A closer look will then yield that in that universe the $\Pi^1_3$-uniformization property fails.

\subsection{$\infty$-allowable Forcings}

\subsubsection{The coding forcings}

We continue with the construction of the appropriate notions of forcing which we want to use in our proof. The goal is to iteratively shrink the set of notions of forcing we want to use until we reach a fixed point. All forcings will belong to a certain class, which we call ``allowable$"$. These are just forcings which iteratively code reals into a generically added \footnote{The idea of generically add the set of $\omega$-blocks of indices of $\vec{S}^1$ or $\vec{S}^2$ is due to S. D. Friedman and D. Schrittesser in their landmark \cite{FS}.  We use it to obtain the closure of allowable forcings under products.} set of $\aleph_1$-many $\omega$-blocks of $\vec{S^1}$ or $\vec{S^2}$.
We first want to present the coding method, which we use to code a real $x$ up, using the definable sequence of $\omega_1$ trees, and subsequently introduce the notion allowable.

Our ground model shall be $W$. Let $x$ be a real,  and let $m,k \in \omega$. We simply write $(x,m,k)$ for a real $w$ which codes the triple $(x,m,k)$ in a recursive way. The forcing $\forceP_{(x,m,k),1}$, which codes the triple $(x,m,k)$ into $\vec{S^1}$ is defined as a two step iteration
 $\forceP_{(x,m,k),1} := (\mathbb{C} (\omega_1))^L \ast \dot{\mathbb{A}}(\dot{Y}) $, where $(\mathbb{C} (\omega_1))^L$ is the usual $\omega_1$-Cohen forcing, as defined in $L$, and $\dot{\mathbb{A}} (\dot{Y})$ is the (name of) an almost disjoint coding forcing, coding a particular set into as real. We shall describe the second factor $\dot{\mathbb{A}} (\dot{Y})$ now in detail.

We let $g \subset \omega_1$ be a $\mathbb{C} (\omega_1)^L$-generic filter over $W$, and let $\rho: [\omega_1]^{\omega} \rightarrow \omega_1$ be some canonically definable, constructible bijection between
these two sets. We use $\rho$ and $g$ to define the set $h \subset \omega_1$, which eventually shall be the set of indices of $\omega$-blocks of $\vec{S}$, where we code up the characteristic function of the real ($(x,y,m)$. Let $h:= \{\rho( g \cap \alpha) \,: \, \alpha < \omega_1 \}$ and let $X \subset \omega_1$ be the $<$-least set  (in some previously fixed well-order of $H(\omega_2)^W[g]$ which codes the following objects:
\begin{itemize}
\item The $<$-least set of $\omega_1$-branches in $W$ through elements of $\vec{S}^1$ which code $(x,y,m)$ at $\omega$-blocks which start at values in $h$, that is  we collect $\{ b_{\beta} \subset S^1_{\beta} \, : \, \beta= \omega \gamma + 2n, \gamma \in h \land n \in \omega \land n \notin (x,y,m) \}$ and  $\{ b_{\beta} \subset S^1_{\beta} \, : \, \beta= \omega \gamma + 2n+1, \gamma \in h \land n \in \omega \land n \in (x,y,m) \}$.
\item The $<$-least set of $\omega_1 \cdot \omega \cdot \omega_1$-many club subsets through $\vec{R}$, our $\Sigma_1 (\{\omega_1\})$-definable sequence of $L$-stationary subsets of $\omega_1$ from the last section, which are necessary to compute every tree $S_{\beta} \in \vec{S}^1$ which shows up in the above item, using the $\Sigma_1 (\{\omega_1\})$-formula from the previous section before Lemma 2.10.
\end{itemize}

Note that, when working in $L[X]$ and if $\gamma \in h$ then
 we can read off $(x,m,k)$  via looking at the $\omega$-block of $\vec{S^1}$-trees starting at $\gamma$ and determine which tree has an $\omega_1$-branch in $L[X]$:
\begin{itemize}
 \item[$(\ast)$]  $n \in (x,m,k)$ if and only if $S^1_{\omega \cdot \gamma +2n+1}$ has an $\omega_1$-branch, and $n \notin (x,m,k)$ if and only if $S^1_{\omega \cdot \gamma +2n}$ has an $\omega_1$-branch.
\end{itemize}
Indeed if $n \notin (x,m,k)$ then we added a branch through $S^1_{\omega \cdot \gamma+ 2n}$. If on the other hand $S^1_{\omega \cdot\gamma +2n}$ is 
Suslin in $L[X]$ then we must have added an $\omega_1$-branch through $S^1_{\omega \cdot \gamma +2n+1}$ as we always add an $\omega_1$-branch through either $S^1_{\omega \cdot \gamma +2n+1}$ or $S^1_{\omega \cdot \gamma +2n}$ and adding branches through some $S^1_{\alpha}$'s  will not affect that some $S^1_{\beta}$ is Suslin in $L[X]$, as $\vec{S}^1$ is independent.

We note that we can apply an argument resembling David's trick \footnote{see \cite{David} for the original argument, where the strings in Jensen's coding machinery are altered such that certain unwanted universes are destroyed. This destruction is emulated in our context as seen below.} in this situation. We rewrite the information of $X \subset \omega_1$ as a subset $Y \subset \omega_1$ using the following line of reasoning.
It is clear that any transitive, $\aleph_1$-sized model $M$ of $\ZFP$ which contains $X$ will be able to correctly decode out of $X$ all the information.
Consequentially, if we code the model $(M,\in)$ which contains $X$ as a set $X_M \subset \omega_1$, then for any uncountable $\beta$ such that $L_{\beta}[X_M] \models \ZFP$ and $X_M \in L_{\beta}[X_M]$:
\[L_{\beta}[X_M] \models \text{\ldq The model decoded out of }X_M \text{ satisfies $(\ast)$ for every $\gamma \in h$\rdq.} \]
In particular there will be an $\aleph_1$-sized ordinal $\beta$ as above and we can fix a club $C \subset \omega_1$ and a sequence $(M_{\alpha} \, : \, \alpha \in C)$ of countable elementary submodels  of $L_{\beta} [X_M]$ such that
\[\forall \alpha \in C (M_{\alpha} \prec L_{\beta}[X_M] \land M_{\alpha} \cap \omega_1 = \alpha)\]
Now let the set $Y\subset \omega_1$ code the pair $(C, X_M)$ such that the odd entries of $Y$ should code $X_M$ and if $Y_0:=E(Y)$ where the latter is the set of even entries of $Y$ and $\{c_{\alpha} \, : \, \alpha < \omega_1\}$ is the enumeration of $C$ then
\begin{enumerate}
\item $E(Y) \cap \omega$ codes a well-ordering of type $c_0$.
\item $E(Y) \cap [\omega, c_0) = \emptyset$.
\item For all $\beta$, $E(Y) \cap [c_{\beta}, c_{\beta} + \omega)$ codes a well-ordering of type $c_{\beta+1}$.
\item For all $\beta$, $E(Y) \cap [c_{\beta}+\omega, c_{\beta+1})= \emptyset$.
\end{enumerate}
We obtain
\begin{itemize}
\item[$({\ast}{\ast})$] For any countable transitive model $M$ of $\ZFP$ such that $\omega_1^M=(\omega_1^L)^M$ and $ Y \cap \omega_1^M \in M$, $M$ can construct its version of the universe $L[Y \cap \omega_1^N]$, and the latter will see that there is an $\aleph_1^M$-sized transitive model $N \in L[Y \cap \omega_1^N]$ which models $(\ast)$ for $w$ and every $\gamma \in h \cap M$.
\end{itemize}
Thus we have a local version of the property $(\ast)$.

In the next step $\dot{\mathbb{A}} (\dot{Y})$, working in $W[g]$, for $g\subset \mathbb{C} (\omega_1)$ generic over $W$, we use almost disjoint forcing $\mathbb{A}_D(Y)$ relative to our previously defined, almost disjoint family of reals $D \in  L $ (see the paragraph after Definition 2.5)  to code the set $Y$ into one real $r$. This forcing only depends on the subset of $\omega_1$ we code, thus $\mathbb{A}_D(Y)$ will be independent of the surrounding universe in which we define it, as long as it has the right $\omega_1$ and contains the set $Y$.

We finally obtained a real $r$ such that
\begin{itemize}
\item[$({\ast}{\ast}{\ast})$] For any countable, transitive model $M$ of $\ZFP$ such that $\omega_1^M=(\omega_1^L)^M$ and $ r  \in M$, $M$ can construct its version of $L[r]$ which in turn thinks that there is a transitive $\ZFP$-model $N$ of size $\aleph_1^M$  such that $N$ believes $(\ast)$ for $w$ and every $\gamma \in h \cap M$.
\end{itemize}
Note that $({\ast} {\ast} {\ast})$ is a $\Pi^1_2$-formula in the parameters $r$ and $w$, as the set $h \cap M \subset \omega_1^M$ is coded into $r$. We will often suppress the reals $r,w$ when referring to $({\ast} {\ast} {\ast})$ as they will be clear from the context. We say in the above situation that the real $w$, which codes $(x,m,k)$\emph{ is written into $\vec{S}^1$}, or that $w$ \emph{is coded into} $\vec{S^1}$ and $r$ witnesses that $w$ is coded. Likewise a forcing $\forceP_{(x,m,k),2}$ is defined for coding the real $w$ which codes $(x,m,k)$ into $\vec{S^2}$.

 The projective and local statement $({\ast} { \ast} {\ast} )$, if true,  will determine how certain inner models of the surrounding universe will look like with respect to branches through $\vec{S}$.
That is to say, if we assume that $({\ast} { \ast} {\ast} )$ holds for a real $w$ and is the truth of it is witnessed by a real $r$. Then $r$ also witnesses the truth of $({\ast} { \ast} {\ast} )$ for any transitive $\ZFP$-model $M$ which contains $r$ (i.e. we can drop the assumption on the countability of $M$).
Indeed if we assume 
that there would be an uncountable, transitive $M$, $r \in M$, which witnesses that $({\ast} { \ast} {\ast} )$ is false. Then by L\"owenheim-Skolem, there would be a countable $N\prec M$, $r\in N$ which we can transitively collapse to obtain the transitive $\bar{N}$. But $\bar{N}$ would witness that $({\ast} { \ast} {\ast} )$ is not true for every countable, transitive model, which is a contradiction.

Consequentially, the real $r$ carries enough information that
the universe $L[r]$ will see that certain trees from $\vec{S}^1$ have branches in that
\begin{align*}
n \in w=(x,y,m) \Rightarrow L[r] \models  ``S^1_{\omega \gamma + 2n+1} \text{ has an $\omega_1$-branch}".
\end{align*}
and
\begin{align*}
n \notin w=(x,y,m) \Rightarrow L[r] \models ``S^1_{\omega \gamma + 2n} \text{ has an $\omega_1$-branch}".
\end{align*}
Indeed, the universe $L[r]$ will see that there is a transitive $\ZFP$-model $N$ which believes $(\ast)$ for every $\gamma \in h \subset \omega_1$, the latter being coded into $r$. But by upwards $\Sigma_1$-absoluteness, and the fact that $N$ can compute $\vec{S}^1$ correctly, if $N$ thinks that some tree in $\vec{S^1}$ has a branch, then $L[r]$ must think so as well.

\subsubsection{Allowable forcings}

Next we define the set of forcings which we will use in our proof.
We aim to iterate the coding forcings we defined in the last section. As the first factor is always $(\mathbb{C}(\omega_1))^L$, the iteration we aim for is actually a hybrid of an iteration and a product. We shall use a mixed support, that is we use countable support on the  product-like coordinates which use $(\mathbb{C}(\omega_1))^L$, and finite support on the iteration-like coordinates which use almost disjoint coding forcing.
\begin{definition}
A mixed support iteration $\forceP=(\forceP_{\beta}\,:\, {\beta< \alpha})$ is called allowable (or 0-allowable, to anticipate later developments) if $\alpha < \omega_1$ and there exists a bookkeeping function $F: \alpha \rightarrow H(\omega_2)^2$ such that 
 $\forceP$ is defined inductively using $F$ as follows:
 \begin{itemize}
 \item If $F(0)=(x,i)$, where $x$ is a real, $i\in \{1,2\}$, then $\forceP_0= \forceP_{x,i}$. Otherwise $\forceP_0$ is the trivial forcing.
 \item Assume that $\beta>0$ and $\forceP_{\beta}$ is defined, $G_{\beta} \subset \forceP_{\beta}$ is a generic filter over $W$. Moreover assume that $F(\beta)=(\dot{x}, i)$, where $\dot{x}$ is a $\forceP_{\beta}$-name of a real, $i \in \{1,2\}$ and $\dot{x}^{G_{\beta}}=x$. Then let $\forceP (\beta)= \forceP_{x,i}= \mathbb{C}(\omega_1))^L \ast {\mathbb{A}} (Y) $, for the reshaped $Y\subset \omega_1$ as being defined in the last section, and let $\forceP_{\beta+1}= \forceP_{\beta} \times (\mathbb{C} (\omega_1))^L \ast \dot{\mathbb{A}} (\dot{Y} )$.

Otherwise we force with just $(\mathbb{C} (\omega_1))^L$.
 \end{itemize}
 We use finite support on the iteration-like parts where almost disjoint coding is used and countable support on the product-like parts where $\omega_1$-Cohen forcing, as computed in $L$ is used.

\end{definition}
 Informally speaking, a (0-) allowable forcing just decides to code the reals which the bookkeeping $F$ provides into either $\vec{S^1}$ or $\vec{S^2}$. Note that the notion of allowable can be defined in exactly the same way over any $W[G]$, where $G$ is a $\forceP$-generic filter over $W$ for an allowable forcing.

We also add that we could have defined allowable in an equivalent way if  we first added, over $W$,  $\omega_1$-many Cohen subsets of $\omega_1$ $\vec{C}=(C_\alpha \, : \, \alpha < \omega_1)$ (in fact we don't really need to do this as we already added them when adding the sequence of $\omega_1$-trees $\vec{S}^1$ and $\vec{S}^2$) with a countably supported product, and then define allowable over the new ground model $W[\vec{C} ]$ as just a finitely supported iteration of almost disjoint coding forcings which select at each step injectively one element $C$ from $\vec{C}$ and the real given by the bookkepping $F$ and the $i \in \{1,2\}$ and then code up all the branches of the trees from $\vec{S}^1$ or $\vec{S}^2$ according to the real $x$ we code for every $\omega$-block with starting value in $h \subset \omega_1$ derived from $C$ as in the last section. That is to say, we could have moved the product factors in an iteration of allowable forcings right at the beginning of our iteration, which we are allowed to do anyway, as it is a product. 
Our current and equivalent approach is a bit easier in terms of notation so we defined allowable the way we did.
 
 We obtain the following first properties of allowable forcings:
 \begin{lemma}

\begin{enumerate}
\item If $\forceP=(\forceP(\beta) \, : \, \beta < \delta) \in W$ is allowable then for every $\beta < \delta$, $\forceP_{\beta} \Vdash| \forceP(\beta)|= \aleph_1$, thus every factor of $\forceP$ is forced to have size $\aleph_1$.
\item Every allowable forcing over $W$ preserves $\omega_1$.
\item The product of two allowable forcings is allowable again.
\end{enumerate}
\end{lemma}
\begin{proof}
The first assertion follows immediately from the definition.

To see the second item we exploit some symmetry.
Indeed, every allowable $\forceP = \Asterisk_{\beta < \delta} P(\beta)= \Asterisk_{\beta < \delta} ( ((\mathbb{C} (\omega_1))^L \ast \dot{\mathbb{A}} (\dot{Y_{\beta} }) )$ can be rewritten as $(\prod_{\beta < \delta}  (\mathbb{C} (\omega_1))^L  )\ast (\Asterisk_{\beta < \delta}  \dot{\mathbb{A}}_D (\dot{Y}_{\beta} ))$ (again with countable support on the  $(\mathbb{C} (\omega_1))^L$ part and finite support on the almost disjoint coding forcings). The latter representation is easily seen to be of the form $\forceP \ast (\Asterisk_{\beta < \delta}  \dot{\mathbb{A}}_D(\dot{Y}_{\beta} ))$, where $\forceP$ is $\sigma$-closed and the second part is a finite support iteration of ccc forcings, hence $\omega_1$ is preserved.

To see that the third item is true, we note that the definition of almost disjoint coding forcing only depends on the subset of $\omega_1$ we want to code and is independent of the surrounding universe $V \supset W$ over which it is defined as long as $Y \in  V$. In particular, if $(Y_{\alpha} \subset \omega_1 \,: \, \alpha < \beta)$ is a sequence of subsets of $\omega_1$ in some ground model, then the finitely supported iteration $\Asterisk_{\alpha< \beta} \dot{\mathbb{A}} (\check{Y}_{\alpha} )$ is isomorphic to the finitely supported product $\prod_{\alpha < \beta} \mathbb{A} (Y_{\alpha} )$.
So we immediately see that if
\[ \forceP^1 = \Asterisk_{\beta < \delta^1} P(\beta)= \prod_{\beta < \delta^1} ( ((\mathbb{C} (\omega_1))^L)  \Asterisk_{\beta < \delta^1} \dot{\mathbb{A}} (\dot{Y_{\beta} }) ) \]
and
\[ \forceP^2 = \Asterisk_{\beta < \delta^2} P(\beta)= \prod_{\beta < \delta^2} ( ((\mathbb{C} (\omega_1))^L \Asterisk_{\beta < \delta^2} \dot{\mathbb{A}} (\dot{Y_{\beta} }) )\]
then
\[ \forceP^1 \times \forceP^2 = \prod_{\beta < \delta^1 + \delta^2} ((\mathbb{C} (\omega_1))^L \Asterisk_{\beta <\delta^1} \dot{\mathbb{A}} (\dot{Y_{\beta} }) )  \Asterisk_{\beta < \delta^2} \dot{\mathbb{A}} (\dot{Y}_{\beta}) \]
which is allowable.
\end{proof}

The proof of the second assertion of the last lemma immediately gives us the following:
\begin{corollary}
Let $\forceP= (\forceP(\beta) \, : \, \beta < \delta) \in W$ be an allowable forcing over $W$. Then $W[\forceP] \models \CH$. Further, if $\forceP= (\forceP(\alpha) \, : \, \alpha < \omega_1) \in W$ is an $\omega_1$-length iteration such that each initial segment of the iteration is allowable over $W$, then $W[\forceP] \models \CH$.

\end{corollary}
Let $\forceP= (\forceP(\beta) \, : \, \beta < \delta)$ be an allowable forcing with respect to some $F \in W$.
The set of  (names of) reals which are enumerated by $F$ we call the set of reals coded by $\forceP$. That is, for every $\beta$, if we let $\dot{x}_{\beta}$ be the (name) of a real  listed by $F(\beta)$ and if we let $G \subset \forceP$ be a generic filter over $W$ and finally if we let
$ \dot{x}_{\beta}^G =:x_{\beta}$,  then we say that
$\{ x_{\beta} \, : \, \beta < \alpha \}$ is the set of reals coded by $\forceP$ and $G$ (though we will suppress the $G$).
Next we show, that iterations of 0-allowable forcings will not add unwanted witnesses to the $\Sigma^1_3$-formula $\psi(w,i)$ which corresponds to the formula $({\ast} {\ast} {\ast})$:
\begin{align*}
 \psi(w,i) \equiv \exists r  \forall M (&M \text{ is countable and transitive and } M \models \ZFP \\&\text{ and } \omega_1^M=(\omega_1^L)^M \text{ and }  r, w \in M  \rightarrow M \models \varphi(w,i) )
\end{align*}
where $\varphi(w,i)$ asserts that in $M$'s version of $L[r]$, there is a transitive, $\aleph_1^M$-sized $\ZFP$-model which witnesses that $w$ is coded into $\vec{S}^i$.
\begin{lemma}\label{nounwantedcodes}
If $\forceP \in W$ is allowable, $\forceP=(\forceP_{\beta} \, : \, \beta < \delta)$, $G \subset \forceP$ is generic over $W$ and $\{ x_{\beta} \, : \, \beta < \delta\}$ is the set of reals which is coded by  $\forceP$. Let $\psi(v_0)$ be the distinguished formula from above. Then
in $W[G]$, the set of reals which satisfy $\psi(v_0)$ is exactly 
$\{ x_{\beta} \, : \, \beta < \delta\}$, that is, we do not code any unwanted information accidentally.
\end{lemma}
\begin{proof}
Let $G$ be $\forceP$ generic over $W$. Let $g= (g_{\beta} \, : \, {\beta} < \delta)$ be the set of the $\delta$ many $\omega_1$ subsets added by the $(\mathbb{C} (\omega_1))^L$-part of the factors of $\forceP$. We let $\rho : ([\omega_1]^{\omega})^L \rightarrow \omega_1$ be our fixed, constructible bijection and let $h_{\beta}= \{ \rho (g_{\beta} \cap \alpha) \, : \, \alpha < \omega_1\}$. Note that the family $\{h_{\beta} \,: \, \beta < \delta \}$ forms an almost disjoint family of subsets of $\omega_1$. Thus there is $\alpha < \omega_1$ such that for arbitrary distinct $\beta_1$, $\beta_2 < \delta$,  $\alpha> h_{\beta_1}\cap h_{\beta_2}$  and additionally, assume that $\alpha$ is an index which does not show up in the set of indices of the trees we code with $\forceP$.

We let $S^1_{\alpha} \in \vec{S}^1$. We claim that there is no real in $W[G]$ such that $W[G] \models L[r] \models ``S^1_{\alpha}$ has an $\omega_1$-branch$"$.
We show this by pulling out the forcing $S^1_{\alpha}$ out of $\forceP$. 
Indeed if we consider $W[\forceP]=L[\forceQ^0] [\forceQ^1][\forceQ^2][\forceP]$, and if $S^1_{\alpha}$ is as described already,
we can rearrange this to $W[\forceP]= L [\forceQ^0] [\forceQ'^1 \times S^1_{\alpha} ] [ \forceQ^2] [\forceP] = W[\forceP'] [S^1_{\alpha} ]$, where $\forceQ'^1$ is $\prod_{\beta \ne \alpha}  S^1_{\beta}$ and $\forceP'$ is $\forceQ^0 \ast \forceQ'^1 \ast \forceQ^2 \ast \forceP$.

Note now that, as $S^1_{\alpha}$ is $\omega$-distributive, $2^{\omega} \cap W[\forceP] = 2^{\omega} \cap W[\forceP']$, as $S_{\alpha}$ is still a Suslin tree in $W[\forceP']$ by the fact that $\vec{S}^1$ and $\vec{S}^2$ are independent, and no factor of $\forceP'$ besides the trees from $\vec{S}^1$ and $\vec{S^2}$ used in $\forceP'$ destroys Suslin trees. But this implies that 
\[W[\forceP'] \models \lnot \exists r L[r] \models `` S^1_{\alpha} \text{ has an $\omega_1$-branch}" \]
as the existence of an $\omega_1$-branch through $S^1_{\alpha}$ in the inner model $L[r]$ would imply the existence of such a branch in $W[\forceP']$. Further
and as no new reals appear when passing to $W[\forceP]$ we also get 
\[W[\forceP] \models \lnot \exists r L[r] \models `` S^1_{\alpha} \text{ has an $\omega_1$-branch}". \]

On the other hand any unwanted information, i.e. any $(x,m) \notin \{(x_{\beta}, m_{\beta}) \, : \, \beta < \delta \}$ such that $W[G] \models \psi((x,i,m))$ will satisfy that there is a real $r$ such that
\[n \in (x,i,m) \rightarrow L[r] \models ``S^1_{\omega \gamma+2n+1} \text{ has an $\omega_1$-branch}" \]
and
\[ n \notin (x,i,m) \rightarrow L[r] \models ``S^1_{\omega \gamma+2n} \text{ has an $\omega_1$-branch}". \]
by the discussion of the last subsection  for $\omega_1$-many $\gamma$'s.

But by the argument above, only trees which we used in one of the factors of $\forceP$ have this property, so there can not be unwanted codes on the $\vec{S}^1$-side. But the very same argument shows the assertion also for the $\vec{S}^2$-side. So for our fixed $\alpha$, there is no real $r$ which codes an $\omega_1$ branch over $L[r]$. But any unwanted information would need not only one but even $\aleph_1$-many such $\alpha$'s chosen as above. This shows that there can not be unwanted information in $W[G]$, as claimed.


\end{proof}

We define next a derivative of the class of allowable forcings.Work with $W$ as our ground model.  Inductively we assume that for an ordinal $\alpha$ and an arbitrary bookkeeping function $F \in W$ mapping to $H(\omega_2)^2$, we have already defined the notion of $\delta$-allowable with respect to $F$ for every $\delta < \alpha$, and the definition works uniformly for every model $W[G]$, where $G$ is a generic filter for an allowable forcing. Note that these inductive requirements are met for 0-allowable forcings. Now we aim to define the derivation of the $<\alpha$-allowable forcings which we call $\alpha$-allowable.

\begin{definition}
Let  $\delta  < \omega_1$ then a $\delta$-length iteration $\forceP$ is called $\alpha$-allowable if it is recursively constructed using a bookkeeping function $F: \delta \rightarrow H(\omega_2)^2$, such that for every $\beta < \delta$, $F(\beta)$ is a pair $(F(\beta)_0, F(\beta)_1)$, and two rules at every stage $\beta < \delta$ of the iteration. We assume inductively that we already created the forcing iteration up to $\beta$, $\forceP_{\beta}$ and we let $G_{\beta}$ denote a hypothetical $\forceP_{\beta}$-generic filter over $W$. We shall now define the next forcing of our iteration $\forceP(\beta)$. Using the bookkeeping $F$ we split into two cases.
\begin{enumerate}
\item[(a)] We assume first that the first coordinate of $F(\beta) ,(F(\beta))_0=(\dot{x},m,k)$, where $\dot{x}$ is the $\forceP_{\beta}$-name of a real and $m<k$ are natural numbers. Further we assume that $\dot{x}^{G_{\beta}}=x$, and $W[G_{\beta}] \models x \in A_m \cup A_k$. 
We assume
 that in $W[G_{\beta}]$, the following is true:
\begin{enumerate}
\item[] There is an ordinal $\zeta < \alpha$, which is chosen to be minimal for which
\item[(i)] for every $\zeta$-allowable forcing $\forceQ \in W[G_{\beta}]$  we have that, over $W[G_{\beta}]$:
\begin{align*}
  \forceQ \Vdash  x \in A_m
\end{align*}
In this situation we force with $\forceP_{(x,m,k),2}$ at stage $\beta$, that is, in $W[G_{\beta}]$, we set $\forceP(\beta):=\forceP_{(x,m,k),2}$.
\item[(ii)] If (i) for $\zeta$ is false but the dual situation is true, i.e.
for every $\zeta$-allowable forcing $\forceQ \in W[G_{\beta}]$, we have that $W[G_{\beta}] $ thinks that
\begin{align*}
  \forceQ \Vdash   x \in A_k
\end{align*}
Then we define $\forceP(\beta)$ to be $\forceP_{(x,m,k),1}$.
\end{enumerate}
If both $(a) (i)$ and $(a) (ii)$ are true for the same $\zeta$, then we give case $(a) (i)$ preference, and suppress case $(a) (ii)$.


\item[(b)] Else $F$ guesses where we  code $(x,m,k)$, i.e. we code $(x,m,k)$ into $\vec{S}^{F(\beta)_1}$, provided  $F(\beta)_1 \in \{1,2\}$ (otherwise we decide to code $(x,m,k)$ into $\vec{S}^1$ per default).
\end{enumerate}

This ends the definition of $\forceP$ being $\alpha$-allowable with respect to $F$ at successor stages $\beta+1$. To define the limit stages $\beta$ of an $\alpha$-allowable forcing,
we assume that we have defined already
$(\forceP_{\gamma} \, : \, \gamma < \beta)$ and let the limit
$\forceP_{\beta}$ to be defined as the inverse limit of the $\mathbb{C}^L$-factors we used in the $\forceP_{\gamma}$'s, and the direct limit of the 
factors of $\forceP_{\gamma}$, which correspond to the almost disjoint coding forcings. In other words we use the already mentioned mixed support, that is countable support on the $\omega_1$-Cohen part, and finite support on the almost disjoint coding part.
\end{definition}

We finally have finished the definition of an $\alpha$-allowable forcing relative to a perviously fixed bookkeeping function $F$. In the following  we often drop the reference to $F$ and simply say that some forcing $\forceP$ is $\alpha$-allowable, in which case we always mean that there is some $F$ such that $\forceP$ is $\alpha$-allowable relative to $F$.
 
We briefly describe a typical run through the cases in the definition of $\alpha$-allowable forcings. Given our bookkeeping $F: \delta \rightarrow H(\omega_2)^2$, the according allowable $\forceP= (\forceP(\beta) \, : \, \beta < \delta)$ forcing is constructed such that at every stage $\beta< \delta$  we ask whether there exists for $\zeta=0$ a $\forceQ$ such that (a)(i) becomes true. If not then we ask the same question for (a)(ii). If both are false, we pass to $\zeta=1$, and so on. If (a) (i) or (a) (ii) never applies for any $\zeta < \alpha$, we pass to (b). It is therefore intuitively clear, and will be proved in a moment, that the notion of $\alpha$-allowable has to satisfy more and more requirements as $\alpha$ increases, hence the classes of $\alpha$-allowable forcings should become smaller and smaller. As a further consequence of this, case (a) in the definition becomes easier and easier to satisfy, which leads in turn to more restrictions of how an $\alpha$-allowable forcing can look like.

\begin{lemma}
For any ordinal $\alpha$, the notion $\alpha$-allowable is definable over the universe $W$. 
\end{lemma}

\begin{lemma}
If $\forceP$ is $\beta$-allowable and $\alpha < \beta$, then $\forceP$ is also $\alpha$-allowable. Thus the classes of $\alpha$-allowable forcings become smaller with respect to the subset relation, if $\alpha$ increases.
\end{lemma}
\begin{proof}
Let $\alpha < \beta$, let $\forceP$ be a $\beta$-allowable forcing and let $F$ be the bookkeeping function which, together with the rules (a)+(b) from above determine $\forceP$. We will show that there is a bookkeeping function $F' \in W$ such that $\forceP$ can be seen as an $\alpha$-allowable forcing determined by $F'$. The first coordinate of $F'$ should always coincide with the first coordinate of $F$, i.e. $\forall \gamma ((F(\gamma)_0=F'(\gamma)_0)$. The second coordinate, which determines  which of the  $\vec{S}^i$-sequence is used for coding when in case (b) is defined via simulating the reasoning for a $\beta$-allowable forcing. This means that at every stage $\gamma$ of the iteration, we pretend that we are working with $\beta$-allowable forcings, we do the reasoning described in (a) and (b) for $\beta$-allowable using $F$. If case (a) does apply, and $\forceP(\gamma)$ is some $\forceP_{(x,m,k),1}$, then we simply let $(F'(\gamma))_1=1$.  That is, we let $F'$ simulate the reasoning we would apply if $\forceP$ would be a $\beta$-allowable forcing using $F$, and the forget about $\beta$-allowable and just keep the result of the reasoning. The new bookkeeping $F'$ is definable from $F$, and clearly $\forceP$ is $\alpha$-allowable using $F'$.

\end{proof}

\begin{lemma}
Let $\alpha$ be an arbitrary ordinal, let $F_1, F_2$ be two bookkeeping functions, $F_1: \delta_1 \rightarrow W^2, F_2: \delta_2 \rightarrow W^2$, and let $\forceP^1=(\forceP^1_{\beta} \, : \, \beta < \delta_1)$ and $\forceP^2=(\forceP^2_{\beta} \, : \, \beta < \delta_2)$ be the $\alpha$-allowable forcings one obtains when using $F_1$ and $F_2$ respectively.

Then $\forceP:=\forceP^1 \times \forceP^2$ is $\alpha$-allowable over $W$, as witnessed by some $F: (\delta_1+\delta_2) \rightarrow W^2$, which is definable from $\{F_1,F_2\}$.
\begin{proof}
By induction on $\alpha$. For $\alpha=0$, this follows immediately from the definition of 0-allowable.

Now suppose the Lemma is true for $\alpha$ and we want to show it is true for $\alpha+1$.
Given $F_1$ and $F_2$, we define $F(\gamma):= F_1(\gamma)$ for $\gamma < \delta_1$ and $F(\delta_1+\gamma):= F_2(\gamma)$ for $\gamma< \delta_2$. We claim that
$\forceP=\forceP^1 \times \forceP^2$ is $\alpha+1$-allowable with respect to $F$ over $W$.
This is shown via induction on the stages $\beta< \delta_1 + \delta_2$, i.e. we shall show that for each $\beta < \delta_1+\delta_2$,
 \[ 
 \forceP_{\beta}=  
\begin{cases}
\forceP^1_{\beta} \text{ if $\beta \le \delta_1$} \\
\forceP^1 \times \forceP^2_{\beta - \delta_1} \text{ if $\beta > \delta_1$}
\end{cases}
\]
is $\alpha+1$-allowable over $W$.

For $\beta \le\delta_1$, this follows immediately from the fact that
$\forceP_1$ is $\alpha+1$-allowable. 

For $\beta < \delta_2$, we assume by induction hypothesis that $\forceP_{\delta_1+\beta}$ is $\alpha+1$-allowable over $W$, and want to see that also $\forceP_{\delta_1+ \beta} \ast \forceP(\delta_1+ \beta+1)$ is $\alpha+1$-allowable.

We work over the model $W[\forceP^1] [\forceP^2_{\beta}]$.
Assume that $F(\delta_1 + \beta+1)_0 =F_2(\beta+1)_0=(\dot{x},m,k)$, let $x=\dot{x}^{G_{\beta}}$, and assume that, when defining $\forceP(\delta_1+ \beta )$ over $W [\forceP^1] [\forceP^2_{\beta}]$, using the rules for $\alpha+1$-allowable, we are in case (a) (i).  We shall show that in this situation, we are in case (a) (i) at stage $\beta$, when defining $\forceP^2(\beta)$ (again using the rules of $\alpha+1$-allowable) over $W[ \forceP^2_{\beta} ]$.

Indeed, as we are in case (a) (i) when defining $\forceP(\delta_1+ \beta )$ over $W [\forceP^1 ] [ \forceP^2_{\beta} ]$, there is a minimal $\zeta< \alpha+1$ such that for every $\forceQ \in W[\forceP^1 ] [\forceP^2_{\beta}]$, which is $\zeta$-allowable, it holds that $W [\forceP^1] [\forceP^2_{\beta}] \models \forceQ \Vdash x \in A_m$. If we assume for a contradiction, that we are not in case (a) (i), 
at stage $\beta$ when defining $\forceP^2(\beta)$ over $W[\forceP^2_{\beta}]$, then there is a $\forceR \in W[\forceP^2_{\beta}]$ such that
$\forceR$ is $\zeta$-allowable and $\forceR \Vdash x \notin A_m$.

But now, by induction hypothesis, $\forceR \in W[\forceP^1 ][\forceP^2_{\beta}]$ is $\zeta$-allowable there as well. Indeed, $\forceP^1$ is $\alpha+1$-allowable, hence $\zeta$-allowable over $W$, and so is the iteration $\forceP^2_{\beta} \ast \forceR$.  By induction hypothesis, $(\forceP^2_{\beta} \ast \forceR) \times \forceP^1  $ is $\zeta$-allowable, and so $\forceR$ is, in $W[\forceP^1] [\forceP^2_{\beta}]$, a $\zeta$-allowable forcing which forces $x$ to not belong to $A_m$, by upwards absoluteness of $\Sigma^1_3$-formulas. Hence we can not be in case (a) (i)  at stage $\delta_1 + \beta$, when defining $\forceP$ which is a contradiction.

The dual reasoning yields that if we are in case (a) (ii) at stage $\beta$  in the definition of $\forceP$ using $F$ over $W$, then we must be in case (a) (ii) as well  at stage $\beta$ in the definition of $\forceP^2(\beta)$ over $W[\forceP^2_{\beta}].$

Last, if we are in case (b) at stage $\beta$ in the definition of $\forceP$ using $F$ over $W[ \forceP_{\delta_1 +\beta} ]$, then we shall show that we are in case (b) as well at stage $\beta$ in the definition of $\forceP^2(\beta)$ over $W[\forceP^2_{\beta}]$. Under our assumption, for every $\zeta < \alpha+1$ there are $\zeta$-allowable forcings $\forceR^{\zeta}_1$ and $\forceR^{\zeta}_2 \in W[\forceP_{\delta_1 +\beta}]$ such that
$\forceR^{\zeta}_1 \Vdash x \notin A_m$ and $\forceR^{\zeta}_2 \Vdash x \notin A_k$.

But by induction hypothesis, $\forceP^1 \ast \forceR^{\zeta}_i$ is $\zeta$-allowable over $W[\forceP^2_{\beta}]$, hence these forcings show that we are in case (b) at stage $\beta$ in the definition of $\forceP^2(\beta)$ over $W[\forceP^2_{\beta}]$.

To summarize, the above shows that if we define the $\alpha+1$-allowable forcing $\forceP$ with $F$ as our bookkeeping function, the outcome will be $\forceP^1 \times \forceP^2$, so the latter is indeed $\alpha+1$-allowable. 

Finally if $\alpha$ is a limit ordinal, then $\forceP^1 \times \forceP^2$ will be $\xi$-allowable for every $\xi < \alpha$, but this implies that $\forceP^1 \times \forceP^2$ is $\alpha$-allowable.
\end{proof}
\end{lemma}

\begin{lemma}
For any $\alpha$, the set of $\alpha$-allowable forcings is non-empty.
\end{lemma}
\begin{proof}
By induction on $\alpha$. If there are $\alpha$-allowable forcings, then the rules (a) and (b) above, together with some bookkeeping $F$ will create an $\alpha+1$-allowable forcing. For limit ordinals $\alpha$, an $\alpha$-allowable forcing always exists, as for any given bookkeeping function $F$ there will be an $\alpha$-allowable, non-trivial forcing with $F$ as its bookkeeping function. 
\end{proof}

As a direct consequence of the last two observations we obtain that there must be an ordinal $\alpha$ such that for every $\beta> \alpha$, the set of $\alpha$-allowable forcings must equal the set of $\beta$-allowable forcings. Indeed every allowable forcing is an $\aleph_1$-sized partial order, thus there are only set-many of them, and the classes (which in fact are sets, if we allow ourselves to identify isomorphic forcings) of $\alpha$-allowable forcings must eventually stabilize at a set which also must be non-empty.

\begin{definition}
Let $\alpha$ be the least ordinal such that for every $\beta> \alpha$, the set of $\alpha$-allowable forcings is equal to the set of $\beta$-allowable forcings. We say that some forcing $\forceP$ is $\infty$-allowable if and only if it is $\alpha$-allowable. Equivalently, a forcing is $\infty$-allowable if it is $\beta$-allowable for every ordinal $\beta$.
\end{definition}
The set of $\infty$-allowable forcings can also be described in the following way. An $\delta < \omega_1$-length iteration $\forceP= (\forceP_{\alpha} \, : \, \alpha< \delta)$ is $\infty$-allowable if it is recursively constructed according to a bookkeeping function $F$ as follows: For every $\beta< \delta$ of the iteration:
\begin{enumerate}
\item[(a)] If the first coordinate of $F(\beta) ,(F(\beta))_0=(\dot{x},m,k)$, where $\dot{x}$ is the $\forceP_{\beta}$-name of a real. Further we assume that $\dot{x}^{G_{\beta}}=x$, for $G_{\beta}$ a $\forceP_{\beta}$-generic filter over $W$ and $W[G_{\beta}] \models x \in A_m \cup A_k$. 
We assume
 that in $W[G_{\beta}]$, the following is true:
\begin{enumerate}
\item[] There is an ordinal $\zeta$, which is chosen to be minimal for which
\item[(i)] 
in the universe $W[G_{\beta}]$, the following holds:
\begin{align*}
 \forall \forceQ &(\forceQ \text{ is } \zeta\text{-allowable}
 \rightarrow \forceQ \Vdash x \in A_m)
\end{align*}
\item[(ii)]
 Or if (a) (i) is not true, but it holds in $W[G_{\beta}]$ that
\begin{align*}
 \forall \forceQ &(\forceQ \text{ is } \zeta\text{-allowable} \rightarrow
 \forceQ \Vdash x \in A_k)
\end{align*}
\end{enumerate}
We give case (a) (i) preference over (a) (ii) if both are true for the minimal $\zeta$.
If this is the case, then we define the $\beta$-th factor of our iteration as $\forceP(\beta):= \forceP_{(x,m,k),2}$ if (a) (i) is true. We let
$\forceP(\beta):= \forceP_{(x,m,k),1}$ if case (a) (ii) is true.

\item[(b)] Otherwise, we let $F(\beta)_1\in \{1.2\}$ decide which $\vec{S}^i$-sequence to use and define
$\forceP(\beta):= \forceP_{(x,m,k),F(\beta)_1}$.
\end{enumerate}

The next Lemma follows immediately from the definitions of $\infty$-allowable and tells us, when an iteration results in an $\infty$-allowable notion of forcing.
 \begin{lemma}\label{allowability}
Let $\delta < \omega_1$ and let  $(\forceP_{\beta} \, : \, \beta < \delta )$ be an $\infty$-allowable forcing over $W$.
Let $\delta' < \omega_1$,  $(\forceQ_{\beta} \mid \beta < \delta') \in W[G_{\delta}]$ be such that
$W[G_{\delta}]  \models  (\forceQ_{\beta} \mid \beta < \delta')$ is $\infty$-allowable.
Then $(\forceP_{\beta} \, : \, \beta < \delta ) \ast (\dot{\forceQ}_{\beta} \mid \beta < \delta')  $ is $\infty$-allowable, over $W$. 
\end{lemma}
As a consequence, blocks of $\infty$-allowable iterations $(\forceP^i_{\beta} \mid \beta < \delta^i)$, $i < \eta <\omega_1$ can be concatenated to one $\infty$-allowable forcing over $W$.
This will be used to see that the upcoming iteration is indeed an $\infty$-allowable iteration over $W$.
\subsection{Definition of the universe in which the ${\Pi^1_3}$ reduction property holds}
The notion of $\infty$-allowable will be used now to define the universe in which the ${\Pi^1_3}$-reduction property is true. We let $W$ be our ground model and start an $\omega_1$-length iteration such that each initial segment of that iteration, which of course must have countable length, is an $\infty$-allowable forcings.
The iteration is guided by a bookkeeping $F: \omega_1 \rightarrow H(\omega_1)^2$, which, on its first coordinate should have the property that for every $(a,b) \in H(\omega_1)^2$, the pre-image $F^{-1} (a,b)$ is unbounded in $\omega_1$.

The definition of the iteration $\forceP= (\forceP_{\beta} \, : \, \beta \le \omega_1)$ is as always specified by induction. At the last stage $\omega_1$ we take the direct limit of the previously construced $\forceP_{\beta}$. Now to the definition of the $\forceP_{\beta}$'s:

\begin{enumerate}
\item  We assume that we are at stage $\beta < \omega_1$, the $\infty$-allowable forcing $\forceP_{\beta}$ has been defined.
Assume that $F(\beta)=(\beta_0,\beta_1) \in H(\omega_1)^2$ and $\beta_0, \beta_1 < \beta$.
We further assume that the $\beta_1$-th (in some previously fixed well-order $<$ of $H(\omega_2))$ $\forceP_{\beta_0}$-name of a triple of the form $(\dot{a},\dot{n}, \dot{l})$, where $\dot{a}$ is a nice $\forceP_{\beta_0}$-name of a real, and $\dot{n}, \dot{l}$ are nice $\forceP_{\beta_0}$-names of natural numbers, is $(\dot{x},\dot{m},\dot{k})$. We assume $\dot{x}^{G_{\beta}}=x$, $\dot{m}^{G_{\beta}}=m$, $\dot{k}^{G_{\beta}}= k$, $k<m$ and that in $W[G_{\beta}]$, $x \in A_m \cup A_k$. 
Now,
if in the universe $W[G_{\beta}]$,
there is a minimal $\zeta < \alpha$ such that
\begin{enumerate}
\item[(i)]$ W[G_{\beta}] \models \forall \forceQ \in W[G_{\beta}] (\forceQ \text{ is $\zeta$-allowable } \rightarrow \forceQ \Vdash x \in A_m)
$, then force with $\forceP(\beta):= \forceP_{(x,m,k),2}$.

Note that this has as a direct consequence, that if we restrict ourselves from now on to forcings $\forceQ \in W[G_{\beta+1}]$ such that $ \forceQ$ is $\zeta$-allowable, then $x$ will remain an element of $A_m$. In particular, the pathological situation that $x \notin A_m$, $x \in A_k$ while $x$ is coded into $\vec{S^2}$ is ruled out for $(x,m,k)$.

\item[(ii)] If we can kick $x$ out of $A_m$ with a $\zeta$-allowable forcing over $W[G_{\beta}]$, yet it is true that
\begin{align*}
  W[G_{\beta}] \models \forall \forceQ \in &W[G_{\beta}]( \forceQ \text{ is $\zeta$-allowable } \rightarrow x \in A_k)
\end{align*}
then force with $\forceP(\beta):=\forceP_{(x,m,k),1}$.

\end{enumerate}

\item If $F(\beta)=(x,m,k)$ and $W[G_{\beta}]\models x \in A_m \cap A_k$ and neither case 1 (i) nor 1 (ii) applies, then we obtain that
\begin{align*}
W[G_{\beta}] \models \exists \forceQ (\forceQ \text{ is } \infty\text{-allowable) and }  \forceQ \Vdash x \notin A_m).
\end{align*}
With the same argument we also obtain that
\begin{align*}
W[G_{\beta}] \models \exists \forceR (\forceR \text{ is } \infty\text{-allowable) and }  \forceR \Vdash x \notin A_k).
\end{align*}
In this situation, we let $\forceQ$ and $\forceR$ the $<$-least $\infty$-allowable forcings as above and use
\[ \forceP(\beta):= \forceQ \times \forceR \]
which is an $\infty$-allowable forcing over $W[G_{\beta}]$ and which forces that $x \notin A_m \cup A_k$.

\end{enumerate}
This ends the definition of the iteration and we shall show that, if $G_{\omega_1}$ denotes a generic filter for the forcing $\forceP_{\omega_1}$, which is defined as the direct limit of the forcings $\forceP_{\beta}$, then the resulting universe $W[G_{\omega_1}]$ satisfies the $\Pi^1_3$-reduction property.
For every pair $(m,k) \in \omega^2$, we define 
\[D^1_{m,k}:= \{ x \in 2^{\omega} \, : \, (x,m,k)\text{ is  not coded into the } \vec{S^1}\text{-sequence}\}\]
and 
\[D^2_{m,k}:= \{ x \in 2^{\omega} \, : \,(x,m,k) \text{ is not coded into the } \vec{S^2}\text{-sequence}\}.\]
Our goal is to show that for every pair $(m,k)$ the sets $D^1_{m,k} \cap A_m$ and $D^2_{m,k} \cap A_k$ reduce the pair of $\Pi^1_3$-sets $A_m$ and $A_k$. 

\begin{lemma}
In $W[G_{\omega_1}]$, for every pair $(m,k)$, $m,k \in \omega$ and corresponding $\Pi^1_3$-sets $A_m$ and $A_k$:
\begin{enumerate}

\item[(a)] $D^1_{m,k} \cap A_m$ and $D^2_{m,k} \cap A_k$ are disjoint.
\item[(b)] $(D^1_{m,k} \cap A_m) \cup (D^2_{m,k} \cap A_k)= A_m \cup A_k$.
\item[(c)] $D^1_{m,k} \cap A_m$ and $D^2_{m,k} \cap A_k$ are $\Pi^1_3$-definable.
\end{enumerate}
\end{lemma}
\begin{proof}
We prove (a) first. If $x$ is an arbitrary real in $A_m \cap A_k$ there will be a least stage $\beta$, such that $F$ at stage $\beta$ considers a triple of names which itself corresponds to the triple  $(x,m,k)$. As $x \in A_m \cap A_k$, we know that case 1 (i)  or 1 (ii) must have applied. We argue for case 1 (i) as case (ii) is similar. In case 1 (i), $\forceP_{(x,m,k),2}$ does code $(x,m,k)$ into $\vec{S^2}$, while ensuring that for all future $\infty$-allowable extensions, $x $ will remain an element of $A_m$.  The rules of the iteration also tell us that $(x,m,k)$ will never be coded into $\vec{S^1}$ by a later factor of the iteration. Thus $x \in D^1_{m,k} \cap A_m$. As we coded $(x,m,k)$ into $\vec{S}^2$, it follows that $x \notin D^2_{m,k}$ and $D^1_{m,k} \cap A_m$ and $D^2_{m,k} \cap A_k$ are disjoint.

To prove (b), let $x$ be an arbitrary element of $A_m \cup A_k$. Let $\beta$ be the stage of the iteration where the triple $(x,m,k)$ is considered first. As $x \in A_m \cup A_k$, either case 1 (i) or (ii) were applied at stage $\gamma$.
Assume first that it was case 1 (i). Then, as argued above, $x \in \cap A_m$ will remain true for the rest of the iteration,  and we will never code $(x,m,k)$ into $\vec{S}^1$ at a later stage of our iteration. Hence $x \in A_m \cap D^1_{m,k}$. If at stage $\beta$ case (ii) applied, then $x \in D^2_{m,k} \cap A_k$, and again, we will never code $(x,m,k)$ into $\vec{S}^1$ at a later stage of our iteration. Thus,either $x \in D^1_{m,k} \cap A_m$ or $x \in D^2_{m,k} \cap A_k$ and we are finished.

To prove (c), we claim that $D^1_{m,k}$ has uniformly the following $\Pi^1_3$-definition over $W[G_{\omega_1}],$ where the $\Sigma^1_3$-formula $\psi((x,m,k),i)$ is defined right above Lemma \ref{nounwantedcodes} (note that in the formulation there the real $w$ is a recursive code for the triple $(x,m,k)$):
\begin{align*}
x \in D^1_{m,k} \cap A_m \Leftrightarrow x \in A_m  \land& \lnot(\exists r (\psi( (x,m,k),1))
\end{align*} 
It is straightforward to see that the right hand side of the equivalence above is the conjunction of two $\Pi^1_3$-formulas, so $\Pi^1_3$ as desired. 
\end{proof}

\subsection{A $\Pi^1_3$-set which can not be uniformized by a $\Pi^1_3$-function}

The next observations will finish the proof of our main result, namely that there is a variant of the thinning out process before which will produce a universe, again denoted with $W[G_{\omega_1}]$, over which there is a $\Pi^1_3$-set which can not be uniformized by a $\Pi^1_3$ function.  

 We assume, without loss of generality, that in our list of $\Pi^1_3$-formulas, the first formula $\varphi_0$ has the following form: \[\varphi_0(x,y) \equiv x=x \land y=y.\] It is clear that there is no allowable, indeed no forcing at all which kicks pair $(x,y)$ out of $A_0$. 
As a consequence, the $\Pi^1_3$-reduction property is automatically satisfied for any pair of $\Pi^1_3$-sets $(A_m,A_0)$. This gives us some freedom to use coding forcings of the form $\operatorname{Code} (x,y,0,m,i)$ in a different way than in our definition of $\alpha$-allowable before. We shall use these forcing now to produce a universe $W[G_{\omega_1}]$ where the $\Pi^1_3$-reduction property holds and additionally the following property holds true.
\begin{itemize}
    \item[$(\star)$] Every \emph{total} $\Pi^1_3$-graph in the plane is also $\Sigma^1_3$.

\end{itemize}
The above property $(\star)$ implies that $\Pi^1_3$-uniformization must fail. Indeed it is straightforward
to see that there always exist $\Pi^1_3$-sets in the plane, with full projection on the first coordinate and which can not be uniformized by a $\Sigma^1_3$-function. Thus, $\Pi^1_3$-uniformization in a universe with the property $(\star)$ is impossible.

To define the desired universe we need to slightly change our thinning out process. Again $F$ will be our bookkeeping function. 
We define the new class of $\alpha$-allowable forcings, denoted by $\Gamma_\alpha$, by induction on $\alpha$:

\begin{definition}[Successor Steps, $\alpha = \delta + 1$]
Assume $\Gamma_\delta$ is defined. At stage $\beta$ of an iteration, if the bookkeeping function $F(\beta)$ yields a tuple of names
$(\dot{x}, \dot{m},\dot{k})$ which are not necessarily full $\forceP_{\beta}$-names but can also be $\forceP_{A}-$-names for $A \subset \beta$ and $\forceP_{A} \Vdash \dot{k} \ne 0$ then we do exactly the same things as before in our old definition of allowable. If on the other hand 
$(\dot{x}, \dot{z}, \dot{m}, 0)$, we evaluate the section $A_m(x) = \{ y \mid \varphi_m(x,y) \}$ with the help of our generic $G_{\beta}$ and consider the following cases.
\begin{itemize}
    \item \textbf{Preservation:} If there exists at least one real $y \in W[G_{\beta}]$ such that $\varphi_m(x,y)$ remains true in \emph{all} further $\Gamma_\delta$-generic extensions, we fix the $<_W$-least such a $y$ (we minimize on the names and then pick the associated $y$). Then we force with $\operatorname{Code}(x,y,m,0,1)$.
    \item \textbf{Guessing:} If no such $y$ exists (i.e., for every $y$, there is a $\Gamma_\delta$-forcing that kills $\varphi_m(x,y)$), the forcing proceeds by adding the bookkeeping real $z$ and forcing with $\operatorname{Code}(x,z,m, \eta)$.
\end{itemize}
A forcing $\mathbb{P}$ is $\alpha$-allowable if it conforms to these constraints relative to $\Gamma_\delta$.
\end{definition}

\begin{definition}[Limit Steps]
For a limit ordinal $\lambda$, a forcing is $\lambda$-allowable if it is $\alpha$-allowable for all $\alpha < \lambda$.
\end{definition}

Now this new notion of $\alpha$-allowable behaves in the very same way than the old notion in that the sets $\Gamma_{\alpha}$ are decreasing if $\alpha$ increases. Also the $\Gamma_{\alpha}$'s are closed under products and will stabilize at some point yielding the non-empty set of $\infty$-allowable forcings $\Gamma_{\infty}$.

This is the right set of forcings for our task. We start, working over $W$ an $\infty$-allowable iteration. The bookkeeping $F$ is defined as in the proof of $\Pi^1_3$-reduction before. In the stages $\beta$ where the bookkeeping $F$ lists some tuple $(\dot{x},\dot{m},\dot{k}$ and $\Vdash \dot{k} \ne 0$, we proceed exactly as in the proof of forcing the $\Pi^1_3$-reduction property.

In the stages $\beta$ where $F(\beta) = (\dot{x},\dot{y},\dot{m},0)$ we do the following:
\begin{itemize}
    \item Assume that $F(\beta)$ evaluates to $(x,y,m,0)$ and $A_m$ at section $x$ has one real $z$ such that $A_m (x,z)$ remains true in all further $\infty$-allowable generic extensions. 
    If this is the case we pick the $<$-least such $z$ and force with
    \[ \operatorname{Code} (x,z,m,0,1).\]
    \item Assume that $F(\beta)$ evaluates to $(x,y,m,0)$ and $A_m$ at section $x$ has no element which will stay in $A_m$ in all future $\infty$-allowable extensions. In this situation we let $\forceP(y)$ be such a forcing which 
    \[ \forceP(y) \Vdash \lnot \varphi_m (x,y)\]
    and force with $\forceP(y)$ at stage $\beta$.
\end{itemize}

The iteration is of length $\omega_1$ of course and we let $W[G_{\omega_1}]$ denote the final model. The proof of the next lemma is the same as before.
\begin{lemma}
    In $W[G_{\omega_1}]$ the $\Pi^1_3$-reduction property holds.
\end{lemma}

The next lemma finishes the proof:
\begin{lemma}
    In $W[G_{\omega_1}]$ if the $\Pi^1_3$-formula $\varphi_m$ defines the set $A_m$ in the plane and $A_m$ it the total graph of a function, then $A_m$ is also $\Sigma^1_3$ via the formula
    \[\{ (x,y,m,0) \mid (x,y,m,0) \text{ is coded into } \vec{S}^1\}\]
\end{lemma}
\begin{proof}
 Suppose that $A_m$ is as in the lemma and $x$ is an arbitrary real. There will be a stage $\beta$  where $(x,z,m,0)$ is listed. As there is exactly one $y$ such that $A_m (x,y)$ holds, we can assume that at stage $\beta$ we will be in case 1 and use $\operatorname{Code (x,y,m,0)}$. In our iteration we will never use a forcing of the form $\operatorname{Code} (x,z,m,0)$ for $z \ne y$ as this would imply that $(x,z) \in A_m$ and $A_m$ is not the graph of a function. So
 $\{ (x,y,m,0) \mid (x,y,m,0) \text{ is coded into } \vec{S}^1\}$ uniformizes $A_m$ as desired.
\end{proof}

One final remark concerning the proof: If $A_m$ is the graph of  a partial $\Pi^1_3$-function then  \[\{ (x,y,m,0) \mid (x,y,m,0) \text{ is coded into } \vec{S}^1\}\] will not uniformize it as at reals $x$ where the $x$-section of $A_m$ is empty, there will be reals $y$ such that $(x,y,m,0,1)$ are coded.


\begin{thebibliography}{9}

\bibitem{Abraham}
U. Abraham \textit{ Proper Forcing}, Handbook of Set Theory Vol.1. Springer
\bibitem{Addison} J. Addison \textit{Some consequences of the axiom of constructibility}, Fundamenta Mathematica, vol. 46 (1959), pp. 337–357.

\bibitem{BHK} J. Baumgartner, L. Harrington and E. Kleinberg \textit{Adding a closed unbounded set.}  Journal of Symbolic Logic, 41(2), pp. 481-482, 1976.


\bibitem{David}
R. David \textit{A very absolute $\Pi^1_2$-real singleton}. Annals of Mathematical Logic 23, pp. 101-120, 1982.

\bibitem{SyVera} 
V. Fischer and S.D. Friedman 
\textit{Cardinal characteristics and projective wellorders}. 
Annals of Pure and Applied Logic 161, pp. 916-922, 2010.

\bibitem{FS}
S. D. Friedman and D. Schrittesser \textit{Projective Measure without Projective Baire}.
Memoirs of the American Mathematical Society vol. 267, 1298. 2020.

\bibitem{FH}
G. Fuchs and J. Hamkins \textit{Degrees of rigidity for Souslin Trees.}
Journal of Symbolic Logic 74(2), pp. 423-454, 2009.

\bibitem{Goldstern} M. Goldstern \textit{A Taste of Proper Forcing.} Di Prisco, Carlos Augusto (ed.) et al., Set theory: techniques and applications. Proceedings of the conferences, Curaçao, Netherlands Antilles, June 26–30, 1995 and Barcelona, Spain, June 10–14, 1996. Dordrecht: Kluwer Academic Publishers. 71-82, 1998.

\bibitem{Ho} S. Hoffelner \textit{$\hbox{NS}_{\omega_1}$ $\Delta_1$-definable and saturated.}
Journal of Symbolic Logic 86(1), pp. 25 - 59, 2021.

\bibitem{H} S. Hoffelner \textit{Forcing the ${\Sigma^1_3}$-separation property.}  Journal of Mathematical Logic 22 (2), 2250008, 2022.
\bibitem{H2} S. Hoffelner \textit{Forcing the $\Pi^1_n$-uniformization property.} Submitted.




\bibitem{JensenSolovay}
R. Jensen and R. Solovay \textit{Some Applications of Almost Disjoint Sets.}
Studies in Logic and the Foundations of Mathematics
Volume 59, pp. 84-104, 1970.

 

\bibitem{Lusin}
N. Lusin \textit{Sur le proble`me de M. J. Hadamard d’uniformisation des ensembles}, Comptes Rendus Acad. Sci. Paris, vol. 190, pp. 349–351.

\bibitem{MS}
D. Martin and J. Steel \textit{A proof of projective determinacy}. Journal of the American  Mathematical  Society pp. 71–125, 1989.

\bibitem{Miyamoto2} T. Miyamoto {\em $\omega_1$-Suslin trees under countable support iterations.} Fundamenta Mathematicae,
vol. 143 (1993), pp. 257–261.


\bibitem{Moschovakis}
Y. Moschovakis \textit{Descriptive Set Theory.} Mathematical Surveys and Monographs 155, AMS.







\end{thebibliography}
\end{document}